\documentclass[11pt, a4paper, reqno]{amsart}

\usepackage{amsmath, amstext, amsthm, amssymb, amsopn, amsxtra, amsfonts, paralist, bbold, setspace, fancyhdr, hyperref, tikz-cd, tkz-euclide}

\usepackage[mathcal]{euscript}
\usepackage[english]{babel}
\usepackage[svgnames]{xcolor}

\makeatletter
\g@addto@macro\normalsize{%
  \setlength\abovedisplayskip{10pt}
  \setlength\belowdisplayskip{10pt}
  \setlength\abovedisplayshortskip{5pt}
  \setlength\belowdisplayshortskip{8pt}
}

\hypersetup{
    colorlinks = true,
    linkcolor = RoyalBlue,
    citecolor = Green,    
}

\newtheorem{thm}{Theorem}[section]
\newtheorem{cor}[thm]{Corollary}
\newtheorem{lem}[thm]{Lemma}
\newtheorem{prop}[thm]{Proposition}

\theoremstyle{definition}
\newtheorem{fact}[thm]{Facts}
\newtheorem{ex}[thm]{Example}
\newtheorem{rmk}[thm]{Remark}
\newtheorem{dfn}[thm]{Definition}

\newcommand{\NN}{\mathbb{N}}
\newcommand{\RR}{\mathbb{R}}
\newcommand{\CC}{\mathbb{C}}
\newcommand{\Born}{\ensuremath{\text{\bf\textsf{Born}}}}

\newcommand{\CBorn}{\ensuremath{\text{\bf\textsf{CBorn}}}}
\newcommand{\NBorn}{\ensuremath{\text{\bf\textsf{NBorn}}}}
\newcommand{\SNrm}{\ensuremath{\text{\bf\textsf{SNrm}}}}
\newcommand{\Tc}{\ensuremath{\text{\bf\textsf{Tc}}}}
\newcommand{\NTc}{\ensuremath{\text{\bf\textsf{NTc}}}}
\newcommand{\TcH}{\ensuremath{\text{\bf\textsf{TcHaus}}}}

\newcommand{\Precon}{\ensuremath{\text{\bf\textsf{Pre}}}}
\newcommand{\Con}{\ensuremath{\text{\bf\textsf{Con}}}}

\newcommand{\TcUbor}{\ensuremath{\text{\bf\textsf{ubTc}}}}
\newcommand{\topp}{\ensuremath{\operatorname{t}}}
\newcommand{\bor}{\ensuremath{\operatorname{b}}}
\newcommand{\colim}{\ensuremath{\operatorname*{colim}}}
\newcommand{\LBb}{\ensuremath{\text{\bf\textsf{LBb}}}}
\newcommand{\bLB}{\ensuremath{\text{\bf\textsf{bLB}}}}
\newcommand{\LB}{\ensuremath{\text{\bf\textsf{LB}}}}
\newcommand{\LBr}{\ensuremath{\text{\bf\textsf{rLB}}}}
\newcommand{\Ban}{\ensuremath{\text{\bf\textsf{Ban}}}}
\newcommand{\LBH}{\ensuremath{\text{\bf\textsf{LBHaus}}}}
\newcommand{\ub}{\ensuremath{\operatorname{ub}}}
\newcommand{\diam}{\hfill\raisebox{-8pt}{\scalebox{0.9}{\rotatebox{45}{$\qed$}}}}

\usepackage{changes}

\begin{document}

\allowdisplaybreaks
$ $
\vspace{-70pt}

\title{Bornological LB-spaces\\[2pt]and idempotent adjunctions}

\author{Jack Kelly\hspace{0.5pt}\MakeLowercase{$^{\text{1}}$}, Lenny Neyt\hspace{0.5pt}\MakeLowercase{$^{\text{2},\,\text{4}}$} and Sven-Ake\ Wegner\hspace{0.5pt}\MakeLowercase{$^{\text{3},\,\text{4}}$}}

\renewcommand{\thefootnote}{}
\hspace{-1000pt}\footnote{\hspace{5.5pt}2020 \emph{Mathematics Subject Classification}: 46M15, 46A17, 46A13, 18E05, 46A08, 46S10, 32B05.\vspace{1.6pt}}


\hspace{-1000pt}\footnote{\hspace{5.5pt}\emph{Key words and phrases}: LB-space, bornological space, bornologification, idempotent adjunction. \vspace{1.6pt}}

\hspace{-1000pt}\footnote{\hspace{0pt}$^{1}$\,Mathematical Institute, University of Oxford, Woodstock Road, Oxford, OX2 6GG, UK.\vspace{1.6pt}}

\hspace{-1000pt}\footnote{\hspace{0pt}$^{2}$\,University of Vienna, Faculty of Mathematics, Oskar-Morgenstern-Platz 1, 1090 Vienna, Austria.\vspace{1.6pt}}

\hspace{-1000pt}\footnote{\hspace{0pt}$^{3}$\,Corresponding author: University of Hamburg, Department of Mathematics, Bundesstra\ss{}e 55, 20146 Ham-\newline\phantom{x}\hspace{1.2pt}burg, Germany, phone: +49\,(0)\,40\:42838\:51\:20, e-mail: sven.wegner@uni-hamburg.de.\vspace{1.6pt}}

\hspace{-1000pt}\footnote{\hspace{0pt}$^{4}$\,This research was funded by the Austrian Science Fund\,(FWF); Project no.~10.55776/ESP8128624, and by\newline\phantom{x}\hspace{1.2pt}the German Research Foundation\,(DFG); Project no.~507660524.}


\begin{abstract}
The notion of an LB-space was introduced by Grothendieck in his 1953 th\`{e}se, referring to a countable colimit of Banach spaces taken within the category of locally convex topological vector spaces, and refining prior work done by Dieudonn\'{e}, Schwartz and K\"othe. Recently, two different notions of `bornological LB-spaces' emerged: one, given by Stempfhuber, refers to countable colimits of Banach spaces as well, but now taken in the category of bornological vector spaces. The other one, given by Bambozzi, Ben-Bassat and Kremnizer, refers to bornologifications of regular LB-spaces, i.e., of LB-spaces in the Grothendieck sense having the additional property that every bounded subset of the colimit is contained and bounded in one of its Banach steps. In this note, we show that the two notions are distinct, but nevertheless closely related. This involves, in particular, an intimate study of the idempotent adjunction of the bornologification and topologification functors.
\end{abstract}

\maketitle

\vspace{-16pt}

\section{Introduction}\label{SEC-1}

The underlying idea of a bornological space is to axiomatise the properties of bounded sets in a topological vector space. Concretely, a \emph{bornology} on a vector space $X$ is a collection $\mathcal{B}$ of subsets of $X$ that covers $X$, is stable under the formation of subsets and finite unions, and which makes the vector space operations bounded in a sense that is straightforward to define. There is a refined notion including convexity requirements. Historically, Mackey \cite{Mackey45} was the first to direct attention to the system of bounded sets and to investigate its relation to a topological vector space structure, see \cite[p.\ 199]{Hogbe-Nlend70}. This led to the notion of a bornological locally convex space. The term `bornology' stems from the French word `born\'e' and appears e.g. in Buchwalter \cite{Buchwalter65}. It seems however that Waelbroeck was the first to give an axiomatic definition; we refer to \cite[pp.\ 198--204]{Hogbe-Nlend70} for more historical notes including references to early lecture notes by Waelbroeck. A systematic study of bornologies was then undertaken in the 1970s by Hogbe-Nlend \cite{HogbeFrench, HN}, Houzel \cite{Houzel}, and Waelbroeck \cite[Chapter~II]{W71}; this also included categorical properties. In 2000, Prosmans and Schneiders \cite{P00, PS00} showed that the bornological vector spaces form a quasi-abelian category and related the latter to Ind-completions of Banach spaces. Around the same time, Meyer employed bornologies in analytic cyclic homology, see \cite{Meyer}. Over the past decade, bornological spaces have found further applications in analytic geometry \cite{B13, B19, B25, BOBB, BBBK18, ben2024perspective}. They are also connected with the formalism of condensed mathematics developed by Clausen and Scholze \cite{Scholze}, see \cite{S25}.

\smallskip

Indeed, bornological vector spaces ($\Born$) often behave better than locally convex topological vector spaces $(\Tc)$, for example, in a filtered colimit of locally convex spaces there can exist sets that are bounded in the colimit topology but which are not the image of a bounded set in one of the steps \--- taking the colimit in the category of bornological vector spaces instead, such a non-regular behaviour cannot happen. One can pass between locally convex and bornological vector spaces via the adjoint pair of bornologification $(\,\cdot\,)^{\bor}$ and topologification $(\,\cdot\,)^{\topp}$ functors, and the fixed points of their compositions, the so-called normal objects, form particularly well-behaved subcategories. Bambozzi, Ben-Bassat and Kremnizer's bornological LB-spaces, which we mentioned in the abstract and will refer to as bLB-spaces ($\bLB$) in this text, account for this paradigm, while Stempfhuber transfers the topological definition directly to the bornological side. We will refer to the latter as LBb-spaces ($\LBb$) and obtain the following picture\vspace{-4pt}
\begin{equation*}
\begin{tikzcd}
 \Born\arrow[r, shift left=0.75ex, "(\,\cdot\,)^{\topp}"]  &\arrow[l, shift right=-0.75ex, "(\,\cdot\,)^{\bor}"]  \Tc\\[2pt]
\LBb\arrow[left hook->]{u}\arrow[two heads]{r}{}&\LB\arrow[hookrightarrow]{u}\\[2pt]
\bLB\arrow[left hook->]{u}\arrow[r, shift left=0.75ex, "\sim"] &\arrow[l, shift right=-0.75ex, "\sim"] \arrow[hookrightarrow]{u}\LBr\hspace{-4pt}
\end{tikzcd}
\vspace{-4pt}
\end{equation*}
illustrating that $(\,\cdot\,)^{\bor}$ and $(\,\cdot\,)^{\topp}$ are quasi-inverses between regular LB-spaces $(\LBr)$ in the topological sense and bornological LB-spaces in the sense of Bambozzi, Ben-Bassat and Kremnizer. While Stempfhuber's bornological LB-spaces map essentially surjectively onto the category of LB-spaces ($\LB$), the latter indeed do not get mapped back to where they came from, leading to the question of what can be said about the range of the functor $(\,\cdot\,)^{\bor}\colon\LB\rightarrow\Born$. In this paper, we will give several answers to this question, which will also yield new characterisations of bLB-spaces, based on classical methods from functional analysis as well as an abstract study of idempotent adjunctions of functors.

\smallskip

Using our characterisations, we will give examples of objects in $\LBr$ and $\bLB$. The DNF spaces of Clausen and Scholze, \cite[Section 8]{Scholze}, that are essentially the $\aleph_{1}$-compact generators of the category of nuclear liquid spaces over $\mathbb{C}$, are in particular regular LB-spaces. Their counterparts on the bornological side are in $\bLB$ and are similarly the $\aleph_{1}$-compact generators of the category of nuclear complete bornological spaces. Sitting further within the category of nuclear complete bornological spaces are \textit{very nuclear} complete bornological spaces \--- these are closely related to generators of the rigidification of the derived (infinity-)category of complete bornological spaces (see \cite{Scholze, kelly2025localising} for more details on the categorical terminology here). Finally, dagger affinoid algebras, the common building blocks of rigid analytic and complex geometry (see \cite{B13, BOBB, B19, B25}), are all very nuclear complete bornological spaces. 

\smallskip

For the sake of readability, we will below often not refer to the original sources that we mentioned at the beginning of this introduction, but try to refer (a) to as few sources as possible and (b) to sources that use similar nomenclature. As the main feature of this note has been studied for the first time in \cite{BBBK18}, we will try to follow their notation where possible. Detailed references for categories of bornological spaces are \cite{W71, Houzel, HN, PS00, Meyer}, background information from functional analysis can be found in \cite{KO69, BPC, MV1997}.


\section{Bornological and locally convex topological vector spaces}\label{SEC-2}

Before we start with the technical details, we would like to warn the reader that the term `bornological space' is used for two different things in different parts of the literature; both definitions will be given below. Moreover, we like to point out that while many authors tacitly assume topological vector spaces to be Hausdorff, we will not do that. Below we give a streamlined summary of notions and results that we will use in the sequel. In order to keep this article reasonably short, we refer the non-expert in the field of locally convex and/or bornological vector spaces to the introductory textbooks mentioned at the end of the introduction as well as to the articles cited below.

\smallskip

Let $k\in\{\RR,\CC\}$ be fixed. We denote by $\Born$ the category of bornological $k$-vector spaces of convex type (short:\ bvs) \cite[Dfn 3.39]{BOBB}. By $\Tc$ we denote the category of, possibly non-Haus\-dorff, locally convex topological $k$-vector spaces (short:\ lcs) and by  $\TcH$ the full subcategory of lcs that are Haus\-dorff (short:\ lcHs) \cite[Dfn 2.1.1 and 3.1.2]{P00}. We write
\begin{equation}\label{FUN}
\begin{tikzcd}
(\,\cdot\,)^{\topp}\hspace{2pt}\colon \Born \arrow[r, shift left=0.75ex]  &\arrow[l, shift right=-0.75ex] \Tc\hspace{2pt}\colon (\,\cdot\,)^{\bor}
\end{tikzcd}
\end{equation}
for the natural functors: the \emph{bornologification} $(\,\cdot\,)^{\bor}$
furnishes $(E,\tau) \in \Tc$ with the von Neumann bornology consisting of all $\tau$-bounded sets, while the \emph{topologification} $(\,\cdot\,)^{\topp}$ endows $(E,\mathcal{B}) \in \Born$ with the topology that has all absolutely convex and bornivorous sets as a basis of $0$-nbhds \cite[p.\ 1881]{BBBK18}. These functors in fact form an adjunction, with $(\,\cdot\,)^{\bor}$ being right adjoint to $(\,\cdot\,)^{\topp}$ by \cite[Prop 2.1.10]{MR961256}.

\smallskip

We denote by $\NTc$, respectively $\NBorn$, the full subcategories of \emph{normal objects}, i.e., lcs such that $(E,\tau)^{\bor\hspace{-1pt}\topp}\cong(E,\tau)$, resp.~bvs such that $(E,\mathcal{B})^{\topp\hspace{-1pt}\bor}\cong(E,\mathcal{B})$ holds \cite[Dfn 3.18]{BBBK18}.

\smallskip
 
We denote by $\CBorn$ the full subcategory of $\Born$ formed by the separated and complete bvs \cite[Dfn 3.43]{BOBB}. Notice that \cite{BOBB} define objects of $\CBorn$ as filtered colimits of Banach spaces in $\Born$ with injective transition maps whereas they were originally defined as bvs in which every bounded set is contained in a bounded Banach disk \cite[Dfn 5.1]{PS00}. Both definitions are equivalent, cf.~Fact \ref{Facts}(vii) below.

\smallskip

As we shall see, $\NTc$ and $\NBorn$ are equivalent to each other, with $(\cdot)^{\bor}$ and $(\cdot)^{\topp}$ restricting to quasi-inverses. By \cite[Thm 2.4.3]{MR961256}, they are both equivalent to the category $\Precon$ of \textit{preconvenient vector spaces} \cite[Dfn 2.4.2]{MR961256}. 

\smallskip

We say that a lcs $(E,\tau)$ is \emph{bornological}, if every absolutely convex and bornivorous subset of $(E,\tau)$ is a 0-nbhd \cite[\S\hspace{1pt}23, 1.5]{FW}. Moreover, we call a lcs $(E,\tau)$ \emph{ultrabornological}, if every absolutely convex subset of $(E,\tau)$ that absorbs any bounded Banach disk is a 0-nbhd (generalisation of \cite[Dfn 4.3(1)]{HN} for possibly non-Hausdorff spaces and consistent with \cite[Dfn 3.2]{MR3577218}). We denote by $\TcUbor$ the full subcategory of $\Tc$ consisting of the ultrabornological lcs.

\begin{fact}\label{Facts} For later use we note the following.\vspace{3pt}

\begin{compactitem}

\item[(i)] \cite[Obs 3.20]{BBBK18} $(\hspace{1pt}\cdot\hspace{1pt})^{\topp}$ commutes with colimits; $(\hspace{1pt}\cdot\hspace{1pt})^{\bor}$ commutes with limits.

\vspace{3pt}

\item[(ii)] \cite[Lem 4.1(1) and (2)]{HN} For any lcs $(E,\tau)$ we have $(E,\tau)^{\bor\hspace{-1pt}\topp\hspace{-1pt}\bor}=(E,\tau)^{\bor}$ and for any bvs $(E,\mathcal{B})$ we have $(E,\mathcal{B})^{\topp\hspace{-1pt}\bor\hspace{-1pt}\topp}=(E,\mathcal{B})^{\topp}$.

\vspace{3pt}

\item[(iii)] [Consequence of (ii)] A lcs $(E,\tau)$ is normal iff $(E,\tau)=(E,\mathcal{B})^{\topp}$ for some bvs $(E,\mathcal{B})$. In particular, $(E,\tau)$ is a bornological lcs iff it is normal.

\vspace{3pt}

\item[(iv)] [Consequence of (i) and (ii)] $\NTc$ is closed under $\Tc$-colimits.

\vspace{3pt}

\item[(v)] \cite[Ch 2, \S2, n$^{\circ}$6]{Houzel} The functors $(\,\cdot\,)^{\topp}\hspace{1pt}\colon \NBorn\hspace{3pt}\raisebox{-1pt}{\scalebox{0.8}{$\stackrel{\textstyle\longrightarrow}{\longleftarrow}$}}\hspace{3pt}\NTc\hspace{1pt}\colon (\,\cdot\,)^{\bor}$ are quasi-inverses.

\vspace{3pt}

\item[(vi)] \cite[Prop 4.1(3)]{HN} or \cite[p.~109]{Houzel} Normed locally convex spaces are normal; for clarity of presentation, we will however always read $(E,\|\cdot\|)$ as a lcs and write $(E,\|\cdot\|)^{\bor}$ if we consider it as a bvs.

\vspace{3pt}

\item[(vii)] \cite[Cor 3.5 and Prop 5.15]{PS00} The category $\Born$ is equivalent to the full subcategory of $\ensuremath{\text{\bf\textsf{Ind(SNrm)}}}$ consisting of essentially monomorphic Ind-objects over semi-normed spaces $\ensuremath{\text{\bf\textsf{SNrm}}}$; similarly, $\CBorn$ is equivalent to essentially monomorphic Ind-objects over Banach spaces, see p.~\pageref{PRES} for more details.


\vspace{3pt}

\item[(viii)] \cite[\S\hspace{1pt}23.5.1 + straightforward modifications]{FW} A lcs is bornological iff it is the filtered $\Tc$-colimit of semi-normed spaces. 
An ultrabornological lcs can always be written as a filtered $\Tc$-colimit of Banach spaces, a filtered $\Tc$-colimit of Banach spaces with injective transition maps is ultrabornological. For lcHs all these properties are well-known to be equivalent.\hfill\diam{}

\end{compactitem}
 
\end{fact}

\begin{rmk}Recall that a \emph{non-trivially valued Banach field} is a field $k$ equipped with a function $|\cdot|\colon k\rightarrow\mathbb{R}_{\geqslant0}$ such that 
\begin{compactitem}
\item[(i)] $|\hspace{1pt}1\hspace{1pt}|=1$,\vspace{2pt}
\item[(ii)] $|\hspace{1pt}0\hspace{1pt}|=0$,\vspace{2pt}
\item[(iii)] $\forall\:x,y\in k\colon |xy|=|x| |y| \text{ and } |x+y|\leqslant |x|+|y|$,\vspace{2pt}
\item[(iv)] $(k,d)$ is separated and complete w.r.t.\ the metric $d(x,y)=|x-y|$, and \vspace{2pt}
\item[(v)] $\exists\:\lambda\in k\backslash\{0\}\colon\displaystyle\lim_{n\rightarrow\infty}|\lambda|^{n}=0$. \vspace{4pt}
\end{compactitem}
As explained in \cite[Section 3]{BBBK18}, the categories $\Tc$, $\Born$, and the adjunction \eqref{FUN}, are defined over any such non-trivially valued Banach $k$. In this generality, the categories $\NBorn$ and $\NTc$ are also defined, see \cite[Dfn 3.18]{BBBK18}.\hfill\diam{} 
\end{rmk}

For most of the remainder of the paper we will only work over $k\in\{\mathbb{R},\mathbb{C}\}$. However, later in Section \ref{SEC:4}, we will interpret several of the Facts \ref{Facts} above purely abstractly in terms of idempotent adjunctions. In this context, we will include results over general non-trivially valued Banach fields $k$.

\smallskip

Filtered colimits of normed spaces with injective transition maps in $\Born$ can only arise as the bornologification of an object from $\Tc$, if the latter does belong to $\TcH$. This follows from the lemma below; we include its simple proof for the sake of completeness. 

\begin{lem}
    \label{l:NecessarilyHausdorff}
    Let $(E,\tau)\in\Tc$ be such that $(E, \tau)^{\bor} \cong \colim_{j \in J} (E_j, \tau_j)^{\bor}$ is a filtered and monomorphic colimit with $(E_j,\tau_j)\in\TcH$ for all $j\in J$. Then $(E, \tau) \in \TcH$.
\end{lem}

\begin{proof}
    Let $D=\overline{\{0\}}^{(E,\tau)}$.
    Then $D \in (E, \tau)^{\bor}$ is bounded and thus for some $j_0 \in J$, the set $D$ is contained and bounded in $(E_{j_0}, \tau_{j_0})$.
    However, $D$ would then be a linear and bounded subspace of a lcHs, which is only possible if $D = \{0\}$.
    Therefore, $(E, \tau)$ is necessarily Hausdorff.
\end{proof}

\smallskip

\section{LB-spaces}\label{SEC-3}

In functional analysis, the subclass of ultrabornological lcs admitting a representation as \emph{countable} colimits of Banach spaces is of particular interest. We follow the notation of  \cite[Dfn 3.24]{BBBK18} and call a lcs $(E,\tau)$ an \emph{LB-space}, if
$$
(E,\tau)\cong\colim_{n\in\NN} (E_n,\|\cdot\|_n)
$$
holds in $\Tc$ with Banach spaces $(E_1,\|\cdot\|_1)\hookrightarrow(E_2,\|\cdot\|_2)\hookrightarrow \cdots$ such that the linking maps are injective. In what follows, we will always understand the symbol `$\hookrightarrow$' between lcs as standing for a linear, continuous and injective map. If we write $(E,\tau)=\colim_{n\in\NN}(E_n,\|\cdot\|_n)\in\LB$ and nothing else is said, we will tacitly assume that the $(E_n,\|\cdot\|_n)$ are a defining sequence of Banach spaces as above. We point out that most of the functional analysis literature requires additionally that $(E,\tau)$ is Hausdorff. That this is not automatic follows from \cite[Prop 2 on p.~177]{M63}, see also Example \ref{MAK} below. A more illustrative example, with only normed spaces as steps though, is due to Waelbroeck \cite[p.~45]{W71} and given by
$$
E_n=\bigl\{f\in\RR[X]\:|\:f(0)=0\bigr\},\;\;\|f\|_n=\hspace{-3pt}\sup_{x\in[0,\frac{1}{n}]}|f(x)|,\,\text{ and }\,\operatorname{id}\colon E_n\hookrightarrow E_{n+1},
$$
which endows $\colim_{n\in\NN}(E_n,\|\cdot\|_n)$ with the trivial topology, see \cite[p.~207--208]{F79} for a proof. 

\smallskip

We denote by $\LB$ the full subcategory of $\Tc$ consisting of all LB-spaces.
If we only consider Hausdorff LB-spaces, we write $\LBH$. Since $\NTc$ is closed under colimits by Fact \ref{Facts}(iv), it follows that $\LB$ is a subcategory of $\NTc$. In particular, we have
\begin{equation}\label{LB-eq}
\forall\:(E, \tau) \in \LB\colon (E, \tau)^{\bor\hspace{-1pt}\topp} = (E, \tau).
\end{equation}


We now review two important facts about objects in $\LBH$. These follow directly from Grothendieck's factorization theorem \cite[Ch I, Cor 1 on p.\ 16]{G55}:

\begin{lem}\label{l:GrothendieckFactorization}\cite[1.4(e) and (f)]{F79} Let $(E, \tau) \cong \colim_{n \in \NN} (E_n, \|\cdot\|_{E_n}) \in \LBH$.\vspace{4pt}

\begin{compactitem}
\item[\emph{(i)}] Suppose that $(E, \tau) \cong \colim_{n \in \NN} (F_n, \|\cdot\|_{F_n})$ holds with possibly different Banach spaces $(F_n,\|\cdot\|_{F_n})$. Then the two defining sequences are \emph{equivalent}, meaning that for every $n\in\NN$ there exists $m\in\NN$ with $m\geqslant n$ and maps
$$
(E_n, \|\cdot\|_{E_n}) \hookrightarrow (F_m, \|\cdot\|_{F_m})\;\;\text{and}\;\; (F_n, \|\cdot\|_{F_n}) \hookrightarrow (E_m, \|\cdot\|_{E_m})
$$
\item[\emph{(ii)}] Let $(X, \|\cdot\|_X)$ be a Banach space such that $(X, \|\cdot\|_X) \hookrightarrow (E, \tau)$ holds. Then, there exists $n \in \NN$ such that already $(X, \|\cdot\|_X) \hookrightarrow (E_n, \|\cdot\|_n)$ holds. \hfill\diam{}
\end{compactitem}
\end{lem}

We follow the notation of \cite[Dfn 3.25]{BBBK18} further and say that an LB-space $(E,\tau)\cong\colim_{n\in\NN}(E_n,\|\cdot\|_n)$ is \emph{regular}, if $(E,\tau)^{\bor}\cong\colim_{n\in\NN} (E_n,\|\cdot\|_{E_n})^{\bor}$ holds for some, or equivalently, every defining sequence of Banach spaces. In view of Lemmas \ref{l:NecessarilyHausdorff} and \ref{l:GrothendieckFactorization}(i), we see that this notion is well-defined and that a regular LB-space is automatically Hausdorff. We write $\LBr$ for the full subcategory of $\TcH$ consisting of all regular LB-spaces; $\LBr$ is strictly smaller than $\LBH$, e.g., by the classic K\"othe-Grothendieck example \cite[31.6, p.~434]{KO69}.

\smallskip

There are two different definitions of what a \emph{bornological LB-space} in $\Born$ should be. To avoid confusion, we introduce the following notation.

\begin{dfn}\label{DFN-LB} Let $(E,\mathcal{B})$ be a bvs.\vspace{3pt}

\begin{compactitem}

\item[(i)] \cite[Dfn 3.26]{BBBK18} $(E,\mathcal{B})$ is a \emph{bLB-space}, if $(E,\mathcal{B})\cong(E,\tau)^{\bor}$ holds with some regular LB-space $(E,\tau)$.

\vspace{4pt}

\item[(ii)] \cite[Dfn 4.3.4]{S25} $(E,\mathcal{B})$ is an \emph{LBb-space}, if $(E,\mathcal{B})\cong\colim_{n\in\NN}(E_n,\|\cdot\|_n)^{\bor}$ holds with a sequence $(E_1,\|\cdot\|_1)\hookrightarrow(E_2,\|\cdot\|_2)\hookrightarrow\cdots$ of Banach spaces.\vspace{4pt}

\end{compactitem}

\noindent{}We denote the corresponding full subcategories of $\Born$ by $\bLB$, resp.~$\LBb$.\hfill\diam{}
\end{dfn}

\begin{rmk}
    In \cite[Dfn 3.26]{BBBK18}, this definition is given in the generality of a non-trivially valued Banach field $k$. 
\end{rmk}

The motivation behind Definition \ref{DFN-LB}(i) is that $(E,\mathcal{B})^{\topp\hspace{-1pt}\bor}=(E,\mathcal{B})$ holds for any bLB-space, compare with \eqref{LB-eq}.
This follows from the independence of regularity for LB-spaces from their defining sequences. Thus, we have $\bLB\subseteq\NBorn$. Our aim now is to compare with Definition \ref{DFN-LB}(ii).

\smallskip

We start by noting that $\bLB\subseteq\LBb$ holds and that we get the following diagram
\begin{equation}\label{DIAG}
\begin{tikzcd}
  \Born\arrow[r, shift left=0.75ex, "(\,\cdot\,)^{\topp}"]  &\arrow[l, shift right=-0.75ex, "(\,\cdot\,)^{\bor}"] \Tc\\[4pt]
\CBorn\arrow[left hook->]{u}\arrow{r}&\arrow[hookrightarrow]{u}\TcUbor\\[4pt]
\LBb\arrow[left hook->]{u}\arrow[two heads]{r}{}&\arrow[hookrightarrow]{u}\LB\\[4pt]
\bLB\arrow[left hook->]{u}\arrow[r, shift left=0.75ex, "\sim"] &\arrow[l, shift right=-0.75ex, "\sim"] \arrow[hookrightarrow]{u}\LBr.\hspace{-4pt}
\end{tikzcd}
\end{equation}
In particular, the range of $(\hspace{1pt}\cdot\hspace{1pt})^{\topp}\colon\LBb\rightarrow\Tc$ is the \emph{whole} category $\LB$ but we observe the following odd effect: if $(E,\tau)=\colim_{n\in\NN}(E_n,\|\cdot\|_n)$ is a non-regular LB-space, then we get via $(E,\mathcal{B}):=\colim_{n\in\NN}(E_n,\|\cdot\|_n)^{\bor}$ a preimage of $(E,\tau)$ under $(\hspace{1pt}\cdot\hspace{1pt})^{\topp}$ in $\LBb$. The bornology $\mathcal{B}$ contains however only sets which are bounded in some $(E_n,\|\cdot\|_n)$, see e.g.~\cite[Prop 2.1]{PS00}, while those sets that caused $(E,\tau)$ to be non-regular have been declared unbounded.

\smallskip

The deeper reason for the above is that $(\hspace{1pt}\cdot\hspace{1pt})^{\bor}$ does not commute with colimits. Indeed, the next example illustrates that one can say almost nothing about the range of $(\hspace{1pt}\cdot\hspace{1pt})^{\bor}\colon\LB\rightarrow\Born$ in general; a restriction to Hausdorff LB-spaces will, however, change the picture entirely, see Theorem \ref{t:RangeLB}.

%
%
%
%
%



\begin{ex}\label{MAK} There exist $(E,\tau)\in\LB$ such that in $(E,\tau)^{\bor}$ every subset of $E$ is bounded. 
\end{ex}
\begin{proof} To prove the existence we apply \cite[Prop 2 on p.\ 177]{M63}. For the sake of completeness, let us include a concrete example and repeat Makarov's argument.

\smallskip

Makarov's construction requires an LB-space $(X, \mu) = \colim_{n \in \NN} (X_n, \|\cdot\|_n)$, such that there exist a sequence $(x_n)_{n \in \NN}\subseteq X_1$, which is dense in $(X, \mu)$, and a sequence $(y_n)_{n \in \NN}\subseteq X$, which is $\mu$-convergent to zero and satisfies $y_n \in X_n$ for $n \in \NN$, but $y_n \notin X_{n-1}$ if $n \geqslant 2$.

\smallskip

Such spaces indeed exist. A concrete example is the classic K\"othe-Grothendieck coechelon space $\lambda^0(A)$ with K\"othe matrix $A$ defined as follows \cite[31.6, p.~434]{KO69}: For $i,j,n\in\NN$ let $a^{\scriptscriptstyle(n)}_{i,j} = j$ for $i \leqslant n$ and $a^{\scriptscriptstyle(n)}_{i,j} = 1$ otherwise. Consider the Banach spaces
\[
X_n = \bigl\{ c = (c_{i, j})_{i, j \in \NN} \in k^{\NN \times \NN}\;\big|\lim_{i, j \to \infty}(a^{\scriptscriptstyle(n)}_{i, j})^{-1}|c_{i, j}| = 0 \bigr\}\;\text{ with }\;\|\hspace{0.2pt}c\hspace{0.8pt}\|_{n} = \sup_{i, j \in \NN}(a^{\scriptscriptstyle(n)}_{i, j})^{-1}|c_{i, j}|
\]
and put $(X,\mu)=\colim_{n \in \NN}(X_n,\|\cdot\|_n)$. For $(x_n)_{n \in \NN}$, we pick some enumeration of all sequences $(c_{i,j})_{i, j \in \NN}$ where only finitely many of the $c_{i,j}$'s are non-zero with rational real and imaginary part. Then we define
\[ 
y_n = (y_{n,i,j})_{i, j \in \NN},\;\;y_{n,i,j} = \begin{cases} 1/n, & i\leqslant n,\\ \hspace{6pt}0, & \text{otherwise.} \end{cases}
\]

What follows is now Makarov's construction applied to $(X, \mu)$, $(x_n)_{n\in\NN}$ and $(y_n)_{n\in\NN}$: put $z_n = x_n + y_n$ for $n \in \NN$ and let $L = \operatorname{span} \{ z_n \: | \: n \in \NN \}$. Then the sequence $(z_n)_{n \in \NN}$ is dense in $(X, \mu)$, hence so is $L$. For $n\in\NN$ put
\[
L_n = L \cap X_n = \operatorname{span} \{ z_1, \ldots, z_n \}
\]
and let $(E, \tau)$ be $X/L$ endowed with the locally convex quotient topology. Each $L_n$ is closed in $(X_n, \|\cdot\|_n)$, which implies
\[
(E, \tau) \cong \colim_{n \in \NN} (X_n / L_n, \|\cdot\|_{X_n / L_n})
\]
and therefore $(E,\tau)\in\LB$. Now, as $L$ is dense in $(X, \mu)$, we have $\overline{\{0\}}^{(E, \tau)}=E=X/L$. Hence, $\tau=\{\varnothing,E\}$ which means that the bornology of $(E,\tau)^{\bor}$ contains every subset of $E$.
\end{proof}

As announced earlier, we restrict our attention now to Hausdorff LB-spaces. Here, we will be able to characterise precisely when an object of $\NBorn$ is the bornologification of some object of $\LBH$.
In fact, these spaces will turn out to be special colimits in $\Born$ of normed spaces.
Therefore, our first result considers exactly when such colimits belong to $\NBorn$.
To do so, we need the following notation.

\smallskip
Let $E$ be a $k$-vector space.
For a non-empty and absolutely convex set $B \subseteq E$, we write $E_B = \operatorname{span} B$. We may endow $E_B$ with the Minkowski functional $\|e\|_B := \inf\{ \lambda > 0 \:|\: e \in \lambda B \}$, turning it into a semi-normed space. If $(E_B, \|\cdot\|_B)$ is a normed space, we call $B$ a \emph{norming disk} in $E$. Notice that $B$ for sure is norming, if $B$ is absolutely convex and $B \in (E, \tau)^{\bor}$ holds for some $(E, \tau) \in \TcH$. If $(E_B, \|\cdot\|_B)$ is even a Banach space, then we call $B$ a \emph{Banach disk} in $E$.

\begin{thm}\label{thm-char}
    Let $(E, \mathcal{B}) = \colim_{n \in \NN} (E_{B_n}, \|\cdot\|_{B_n})^{\bor} \in \Born$ for a sequence of norming disks $(B_n)_{n \in \NN}$ in $E$.
    Then, $(E, \mathcal{B}) \in \NBorn$ if and only if the following two conditions are met for $(B_n)_{n \in \NN}$:
        \begin{itemize}
        \item[(I)] $\forall\:e\in E\backslash\{0\}\;\exists\:e'\colon E\rightarrow k \text{ linear}\colon e'(e) \neq 0 \:\text{ and }
        \sup_{b \in B_n} |e'(b)| < \infty$ for all $n\in\NN$;
            \vspace{3pt}
            \item[(II)] $\forall\:n\in\NN\;\exists\:m\geqslant n\colon \overline{B_n}^{\scriptscriptstyle(E, \mathcal{B})^{\topp}}$ is bounded in $(E_{B_m}, \|\cdot\|_{B_m})^{\bor}$.
        \end{itemize}
\end{thm}

\begin{proof}
    \textquotedblleft{}$\Longrightarrow$\textquotedblright{} Suppose $(E, \mathcal{B})$ is in $\NBorn$.
    Put $(E, \tau) = (E, \mathcal{B})^{\topp} \cong \colim_{n \in \NN} (E_{B_n}, \|\cdot\|_{B_n})$, then $(E, \tau)^{\bor} = (E, \mathcal{B})$.
    By Lemma \ref{l:NecessarilyHausdorff}, $(E, \tau)$ is Hausdorff.
    Using the Hahn-Banach theorem, for every $e \in E\backslash\{0\}$ there exists a continuous linear map $e'\colon(E, \tau) \to (k, |\cdot|)$ such that $e'(e) \neq 0$.
    Fix $n\in\NN$. Then $e'$, restricted to each step $(E_{B_n}, \|\cdot\|_{B_n})$, is a bounded linear map and hence $\sup_{b \in B_n} |e'(b)| < \infty$ holds.
    Therefore, (I) is true.
    Now, each $B_n$ is bounded in $(E, \tau)$, hence its closure in $(E,\tau)=(E, \mathcal{B})^{\topp}$ is bounded, too. 
    Then (II) becomes clear.

    \vspace{3pt}

    \textquotedblleft{}$\Longleftarrow$\textquotedblright{}Assume now that (I) and (II) hold for $(E, \mathcal{B})$.
    By (I), we see that $\{0\}$ is closed in $(E, \tau) := (E, \mathcal{B})^{\topp} = \colim_{n \in \NN} (E_{B_n}, \|\cdot\|_{B_n})$, hence $(E, \tau) \in \TcH$.
    Put now $\widetilde{B}_n = \overline{B_n}^{\scriptscriptstyle(E, \tau)}$ for $n \in \NN$.
    It follows from \cite[Prop 25.16]{MV1997} that
        \[ (E, \tau)^{\bor} = \colim_{n \in \NN} (E_{\widetilde{B}_n}, \|\cdot\|_{\widetilde{B}_n})^{\bor} \]
    holds. By (II), we then have $(E, \tau)^{\bor} \subseteq (E, \mathcal{B})$.
    Therefore, in fact, $(E, \mathcal{B})^{\topp\hspace{-1pt}\bor} = (E, \tau)^{\bor} = (E, \mathcal{B})$ which implies $(E, \mathcal{B}) \in \NBorn$.
\end{proof}

We move on to our bornological characterisation of $\LBH$. 
To achieve this, we need to introduce the notion of pre-Banach disks.

\smallskip

Let $(E, \mathcal{B}) \in \Born$ be given. We write $\widehat{\mathcal{B}}(E)\subseteq \mathcal{B}$ for the set of all Banach disks in $E$ contained in $(E, \mathcal{B})$.
Let $\mathcal{B}_{\ub}$ be the bornology generated by $\widehat{\mathcal{B}}(E)$.
We define
\[
(E, \mathcal{B})^{\ub} := (E, \mathcal{B}_{\ub})^{\topp} . 
\]
Then $(E, \mathcal{B})^{\ub} \in \TcUbor$ and we call it the \emph{ultrabornologification of $(E, \mathcal{B})$}.

\smallskip

\begin{dfn} \label{def:preBanachDisk} Let $(E, \mathcal{B}) \in \Born$. A norming disk $B \in \mathcal{B}$ is called a \emph{pre-Banach disk} in $(E, \mathcal{B})$, if there exists a Banach disk $\widehat{B} \subseteq B$ such that $B$ is contained in the closure of $\widehat{B}$ in $(E, \mathcal{B})^{\ub}$.\hfill\diam{}
\end{dfn}

As the notion of pre-Banach disk is new, we now go a bit more in depth to familiarise the reader with it. The definition of a pre-Banach disk depends on the bornology we are working with. However, it remains unchanged under monomorphisms in $\Born$. In fact, the ensuing argument works for any morphism $\varphi\colon(E, \mathcal{B}) \to (F, \mathcal{C})$ that has a \emph{locally closed kernel}, i.e., 
when $\ker(\varphi_{\mid E_{C}})$ is closed in $(E_C, \|\cdot\|_C)$ for each $C \in \widehat{B}(E)\cup\{B\}$.
In particular, the latter holds when $\ker \varphi$ is finite dimensional or when $(F, \mathcal{C}) \in \CBorn$.


\begin{lem}\label{l:PreBanachDiskInclusion} Let $(E, \mathcal{B}),(F, \mathcal{C}) \in \Born$ be such that $\iota\colon(E, \mathcal{B})\to (F, \mathcal{C})$ is a monomorphism. If $B$ is a pre-Banach disk in $(E, \mathcal{B})$, then $\iota(B)$ is also a pre-Banach disk in $(F, \mathcal{C})$.
\end{lem}

\begin{proof} Let $\widehat{B} \subseteq B$ be a Banach disk in $(E, \mathcal{B})$ such that $B$ is contained in the closure of $\widehat{B}$ in $(E, \mathcal{B})^{\ub}$. Since Banach disks are preserved under monomorphisms in $\Born$, the linear map $\iota : (E, \mathcal{B})^{\ub} \to (F, \mathcal{C})^{\ub}$ is continuous, and it follows that the closure of $\iota(\widehat{B}) \in \widehat{\mathcal{C}}(F)$ in $(F, \mathcal{C})^{\ub}$ also contains $\iota(B)$.
\end{proof}

A Banach disk is a pre-Banach disk in each bornology that contains it. In fact, this is the defining property of Banach disks.

\begin{lem}\label{l:PreBanachDiskSmallest}For a norming disk $B$ in a $k$-vector space $E$ the following are equivalent:\vspace{2pt}
    \begin{compactitem}
        \item[(i)] $B$ is a Banach disk.\vspace{1pt}
        \item[(ii)] For any $(E, \mathcal{B}) \in \Born$ with $B \in \mathcal{B}$, the set $B$ is a pre-Banach disk in $(E, \mathcal{B})$.
    \end{compactitem}
\end{lem}

\begin{proof} (i)\hspace{1pt}$\Longrightarrow$\hspace{1pt}(ii): As $B$ is an element of the bornology of $(E_B, \|\cdot\|_B)^{\bor}$, by choosing $\widehat{B} = B$ in Definition \ref{def:preBanachDisk} it follows that $B$ is a pre-Banach disk in $(E_B, \|\cdot\|_B)^{\bor}$. Then (ii) follows from Lemma \ref{l:PreBanachDiskInclusion}.

\smallskip

(ii)\hspace{1pt}$\Longrightarrow$\hspace{1pt}(i): Using an algebraic complement of $E_B$, we may extend $\|\cdot\|_B$ to a norm $\|\cdot\|$ on $E$.
We consider $(E, \|\cdot\|)^{\bor} \in \Born$, which contains $B$ in its bornology.
Therefore, by our assumption, it follows that $B$ is a pre-Banach disk in $(E, \|\cdot\|)^{\bor}$.
Let $\widehat{B} \subseteq B$ be a Banach disk in $(E, \|\cdot\|)^{\bor}$ for which $B$ is contained in the closure of $\widehat{B}$ in the ultrabornologification of $(E, \|\cdot\|)^{\bor}$.
As the latter topology is finer than that of $(E, \|\cdot\|)$, it also follows that $B$ is contained in the closure of $\widehat{B}$ in $(E, \|\cdot\|)$.
Therefore, as $\|\cdot\|$ coincides with $\|\cdot\|_B$ on $E_B$,
    \[ B \subseteq \overline{\widehat{B}}^{(E, \|\cdot\|)} \cap E_B = \overline{\widehat{B}}^{(E_B, \|\cdot\|_B)} . \]
Take now any $e \in B$. Then, there exists some $e_1 \in \widehat{B}$ such that $e - e_1 \in 2^{-1}B$. Next, we find $e_2 \in 2^{-1}\widehat{B}$ for which $(e - e_1) - e_2 \in 2^{-2}B$. Continuing this process, we construct a sequence $(e_n)_{n \in \NN}$ in $\widehat{B}$ satisfying
\[
e_n \in 2^{-(n-1)} \widehat{B}\;\text{ and }\; e-\mathop{\textstyle\sum}_{j=1}^n e_j\in2^{-n} B
\]
for $n\in\NN$ and we obtain
\[
e = \mathop{\textstyle\sum}_{n \in \NN} e_n \in 2 \widehat{B}.
\]
Therefore, $\widehat{B} \subseteq B \subseteq 2 \widehat{B}$.
    Consequently, $E_B = E_{\widehat{B}}$ and $\|\cdot\|_B$ and $\|\cdot\|_{\widehat{B}}$ are equivalent norms on $E_B$. We conclude that $(E_B, \|\cdot\|_B)$ is a Banach space, so that $B$ is a Banach disk.
\end{proof}

We now present our description of the range of $(\hspace{1pt}\cdot\hspace{1pt})^{\bor}\colon\LBH\rightarrow\Born$. 

\begin{thm}\label{t:RangeLB} For $(E, \mathcal{B}) \in \Born$ the following are equivalent: \vspace{3pt}

\begin{compactitem}

\item[(i)] $(E, \mathcal{B}) = (E, \tau)^{\bor}$ for some $(E, \tau) \in \LBH$.

\vspace{2pt}

\item[(ii)] There exists a sequence of Banach disks $(\widehat{B}_n)_{n \in \NN}$ and a sequence of pre-Banach disks $(B_n)_{n \in \NN}$ in $(E, \mathcal{B})$ such that $(B_n)_{n \in \NN}$ satisfies (I) and (II) of Theorem \ref{thm-char} and
\begin{equation}
\phantom{x\hspace{32pt}x}E = \mathop{\textstyle\bigcup}_{n \in \NN} E_{\widehat{B}_n} \;\,\text{ and }\;\; (E, \mathcal{B}) = \colim_{n \in \NN} (E_{B_n}, \|\cdot\|_{B_n})^{\bor}. \tag{4a \text{\,and\,} 4b}
\begin{picture}(0,0)(0,0)\put(-500,10){\begin{minipage}{20pt}
\begin{equation}\label{eq:EUnionPreBanachDisks}\tag{4a}\vspace{-50pt}
\end{equation}
\begin{equation}\label{eq:LBSpPreBanachDisk}\tag{4b}
\end{equation}
\end{minipage}}
\end{picture}
\end{equation}

\end{compactitem}
\end{thm}

\addtocounter{equation}{1}

\begin{proof}(i)\hspace{1pt}$\Longrightarrow$\hspace{1pt}(ii): Suppose we have $(E, \tau) = \colim_{n \in \NN} (E_n, \|\cdot\|_n)\in\LBH$ with $(E,\mathcal{B})=(E,\tau)^{\bor}$. Let $\widehat{B}_n$ denote the unit ball of $(E_n, \|\cdot\|_n)$, thus in particular \eqref{eq:EUnionPreBanachDisks} holds and each $\widehat{B}_n$ is a Banach disk in $(E, \mathcal{B})$. We will now first show that
\begin{equation}\label{eq:CoLimitIsUB}
(E, \tau) = (E, \mathcal{B})^{\ub}
\end{equation}
holds. We have $E_n=E_{\widehat{B}_n}$ by construction and thus $\colim_{n \in \NN}(E_{\widehat{B}_n}, \|\cdot\|_{\widehat{B}_n})=(E,\tau)$ as vector spaces. It follows $(E_{\widehat{B}_n}, \|\cdot\|_{\widehat{B}_n}) \hookrightarrow (E, \mathcal{B})^{\ub}$ and thus the identity
\begin{equation}\label{eq:ID}
\operatorname{id}\colon\colim_{n \in \NN} (E_{\widehat{B}_n}, \|\cdot\|_{\widehat{B}_n}) \rightarrow (E, \mathcal{B})^{\ub}
\end{equation}
is continuous. As the topology of $((E, \tau)^{\bor})^{\ub}$ is always finer than $\tau$, we conclude that \eqref{eq:ID} is an isomorphism in $\Tc$. This establishes \eqref{eq:CoLimitIsUB}. We now write $B_n$ for the closure of $\widehat{B}_n$ in $(E,\tau)$. Then each $B_n$ is a pre-Banach disk in $(E, \mathcal{B})$ by \eqref{eq:CoLimitIsUB}. The validity of \eqref{eq:LBSpPreBanachDisk} follows from \cite[Prop 25.16]{MV1997}.
Since $(E, \tau) \in \NTc$, also $(E, \mathcal{B}) \in \NBorn$, therefore Theorem \ref{thm-char} shows that $(B_n)_{n \in \NN}$ satisfies conditions (I) and (II).

\smallskip

(ii)\hspace{1pt}$\Longrightarrow$\hspace{1pt}(i): Suppose we have a sequence of bounded Banach disks $(\widehat{B}_n)_{n \in \NN}$ and a sequence of pre-Banach disks $(B_n)_{n \in \NN}$ in $(E, \mathcal{B})$ for which \eqref{eq:EUnionPreBanachDisks} and \eqref{eq:LBSpPreBanachDisk} hold, and such that $(B_n)_{n \in \NN}$ satisfies conditions (I) and (II) of Theorem \ref{thm-char}. For every $n\in\NN$, let $\widetilde{B}_n\subseteq B_n$ be a Banach disk such that $B_n$ is contained in the $(E,\mathcal{B})^{\ub}$-closure of $\widetilde{B}_n$. Potentially by changing $B_n$ into $\sum_{j = 1}^n [B_j + \widehat{B}_j]$ and $\widehat{B}_n$ into $\sum_{j=1}^n [\widetilde{B}_j + \widehat{B}_j]$ 
 -- using that $(E,\mathcal{B})$ is separated by (I) and \eqref{eq:LBSpPreBanachDisk} and thus the argument in \cite[Lem 3.2.10]{W03} may be used to conclude that the Minkowski sum of two Banach disks is again a Banach disk --
we may assume that the sequences $(\widehat{B}_n)_{n \in \NN}$ and $(B_n)_{n \in \NN}$ are increasing and that $B_n$ is contained in the closure of $\widehat{B}_n$ in $(E, \mathcal{B})^{\ub}$ for each $n \in \NN$. Put $(E, \tau) = (E, \mathcal{B})^{\topp}$, then
\[
(E, \tau) = \colim_{n \in \NN} (E_{B_n}, \|\cdot\|_{B_n})
\]
and $(E, \mathcal{B}) = (E, \tau)^{\bor}$ since $(E, \mathcal{B})\in\NBorn$ by Theorem \ref{thm-char}. Note that by Lemma \ref{l:NecessarilyHausdorff}, we have $(E, \tau) \in \TcH$. Now we claim
\[
(E, \mathcal{B})^{\ub}= \colim_{n \in \NN} (E_{\widehat{B}_n}, \|\cdot\|_{\widehat{B}_n}).
\]
Indeed, by \eqref{eq:EUnionPreBanachDisks} we have that the right-hand side coincides as a vector space with $E$. Moreover, a similar argument as before shows that $\colim_{n \in \NN} (E_{\widehat{B}_n}, \|\cdot\|_{\widehat{B}_n}) \hookrightarrow (E, \mathcal{B})^{\ub}$. Both topologies are finer than the Hausdorff topology $\tau$ and thus they are Hausdorff, too. But then, they must necessarily coincide by De Wilde's open mapping theorem \cite[Thm 24.30]{MV1997}. This establishes the claim.

\smallskip

Take now any absolutely convex, closed 0-nbhd $V$ in $(E, \mathcal{B})^{\ub}$. Then $V \cap E_{\widehat{B}_n}$ is a closed 0-nbhd in $(E_{\widehat{B}_n}, \|\cdot\|_{\widehat{B}_n})$ for each $n \in \NN$. Consequently, for any $n \in \NN$, there is some $r_n > 0$ such that 
\[
r_n \widehat{B}_n \subseteq V \cap E_{\widehat{B}_n} \subseteq V.
\]
Taking the closures of the left- and right-hand side in $(E, \mathcal{B})^{\ub}$ gives
\[
r_n B_n \subseteq V.
\]
Therefore, $V$ is also a 0-nbhd of $(E, \tau)$. We conclude  
\[ 
(E, \tau)=(E, \mathcal{B})^{\ub}= \colim_{n \in \NN} (E_{\widehat{B}_n}, \|\cdot\|_{\widehat{B}_n}),
\]
that is, $(E, \tau) \in \LBH$.
\end{proof}



Next, we establish the following characterisation of $\bLB$.

\begin{thm}\label{t:bLB=rangeInLBb}It holds
\begin{align}
\bLB &= \{ (E, \mathcal{B}) \in \CBorn \mid (E, \mathcal{B}) = (E, \tau)^{\bor} \text{ for some } (E, \tau) \in \LB  \}\label{eq:bLB=LBRangeinCBorn}\\
     &= \{ (E, \mathcal{B}) \in \LBb \mid (E, \mathcal{B}) = (E, \tau)^{\bor} \text{ for some } (E, \tau) \in \Tc \}.\label{eq:bLB=rangeInLBb}
\end{align}
\end{thm}

\begin{proof} Any $(E, \mathcal{B}) \in \bLB$ is the image under $(\hspace{1pt}\cdot\hspace{1pt})^{\bor}$ of some $(E, \tau) \in \LBr$, so it is clearly contained in the right-hand sides of \eqref{eq:bLB=LBRangeinCBorn} and \eqref{eq:bLB=rangeInLBb}.

\smallskip

Now we first verify \eqref{eq:bLB=LBRangeinCBorn}. Suppose $(E,\tau) \in \LB$ is such that $(E, \mathcal{B}) = (E, \tau)^{\bor} \in \CBorn$. By Lemma \ref{l:NecessarilyHausdorff}, we have that $(E, \tau)$ is Hausdorff. Let $(E, \tau) = \colim_{n \in \NN} (E_n, \|\cdot\|_n)$. By Lemma \ref{l:GrothendieckFactorization}(ii), each Banach disk in $(E, \mathcal{B})$ is contained in some step $(E_n, \|\cdot\|_n)^{\bor}$.  Now, as $(E, \mathcal{B}) \in \CBorn$, we have that each bounded set of $(E, \tau)$ is contained and bounded in some step. Consequently, $(E, \tau) \in \LBr$.

\smallskip

It remains to prove \eqref{eq:bLB=rangeInLBb}. Let $(E, \mathcal{B}) = (E, \tau)^{\bor} = \colim_{n \in \NN} (E_n, \|\cdot\|_n)^{\bor}$ for a sequence $(E_1,\|\cdot\|_1)\hookrightarrow(E_2,\|\cdot\|_2)\hookrightarrow\cdots$ of Banach spaces.  Then $(E, \mathcal{B})^{\topp} \cong \colim_{n \in \NN} (E_n, \|\cdot\|_n)$. Also 
\[
((E, \mathcal{B})^{\topp})^{\bor} = (E, \mathcal{B}) \in \CBorn.
\]
By our previous observations, we have $(E, \mathcal{B})^{\topp} \in \LBr$, which finishes our argument.
\end{proof}

We may now point out the exact difference between $\bLB$ and $\LBb$.
Essentially, the ensuing result states that $(E, \mathcal{B}) \in \LBb$ belongs to $\bLB$ if in Theorem \ref{t:RangeLB} we may choose $\widehat{B}_n = B_n$.

\begin{thm}
\label{t:bLb=LBb+NBorn}
An $(E, \mathcal{B}) = \colim_{n \in \NN} (E_{B_n}, \|\cdot\|_{B_n})^{\bor} \in \LBb$ with norming disks $B_n\subseteq E$ belongs to $\bLB$ if and only if $(B_n)_{n \in \NN}$ satisfies (I) and (II) of Theorem \ref{thm-char}.
In particular,
    \[ \bLB = \LBb \cap \NBorn . \]
\end{thm}

\begin{proof}
If $(E, \mathcal{B}) \in \bLB$, then in particular $(E, \mathcal{B}) \in \NBorn$, so necessity follows from Theorem \ref{thm-char}.
Conversely, if the conditions are satisfied, Theorem \ref{thm-char} shows $(E, \mathcal{B}) \in \NBorn$. Therefore, $(E, \mathcal{B})$ belongs to the right-hand side of \eqref{eq:bLB=rangeInLBb}, so we have $(E, \mathcal{B}) \in \bLB$ by Theorem \ref{t:bLB=rangeInLBb}.
\end{proof}

Theorem \ref{t:bLB=rangeInLBb} also yields the following description of the preimage of $\LBb$ under the functor $(\hspace{1pt}\cdot\hspace{1pt})^{\bor}\colon\LB\rightarrow\Born$.


\begin{thm}
\label{t:CharRegularity}
For any $(E, \tau) \in \LB$, the following are equivalent:\vspace{2pt}
\begin{compactitem}
\item[(i)] $(E, \tau) \in \LBr$.\vspace{1pt}
\item[(ii)] $(E, \tau)^{\bor} \in \bLB$.\vspace{1pt}
\item[(iii)] $(E, \tau)^{\bor} \in \LBb$.\vspace{1pt}
\item[(iv)] $(E, \tau)^{\bor} \in \CBorn$.
\end{compactitem}
\end{thm}

\begin{proof}
    (i)\hspace{1pt}$\Longleftrightarrow$\hspace{1pt}(ii) holds by definition, while (ii)\hspace{1pt}$\Longrightarrow$\hspace{1pt}(iii)\hspace{1pt}$\Longrightarrow$\hspace{1pt}(iv) is straightforward. Finally, the implication (iv)\hspace{1pt}$\Longrightarrow $\hspace{1pt}(ii) follows from Theorem \ref{t:bLB=rangeInLBb}.
\end{proof}

We point out that several well-known classes of bornological spaces are contained in $\bLB$.

\begin{ex}\label{EX-NUC}
$(E, \tau) \in \LB$ is called an $\text{LS}_{\text{\rm w}}$\emph{-space}, if $(E, \tau) = \colim_{n \in \NN} (E_n, \|\cdot\|_{E_n})$ is such that
    \begin{equation}
        \label{eq:NuclearCond}
        \forall\:n \in \NN\; \exists\:m \geqslant n \colon (E_n, \|\cdot\|_{E_n}) \hookrightarrow (E_m, \|\cdot\|_{E_m}) \text{ is weakly compact}. 
    \end{equation}        
Each $\text{LS}_{\text{\rm w}}$-space $(E, \tau)$ is an element of $\LBr$ \cite[Cor 8.5.25(i)]{BPC}. By Lemma \ref{l:GrothendieckFactorization}(i), this definition is independent of the defining sequence of Banach spaces for $(E, \tau)$.

\smallskip

Similarly, we call an element $(E, \mathcal{B}) \in \LBb$ a \emph{bornological} $\text{LS}_{\text{\rm w}}$\emph{-space} if we have $(E, \mathcal{B}) = \colim_{n \in \NN} (E_n, \|\cdot\|_{E_n})^{\bor}$ for Banach spaces satisfying \eqref{eq:NuclearCond}.
Clearly, if $(E, \mathcal{B}) \in \LBb$ is a bornological $\text{LS}_{\text{\rm w}}$-space, then $(E, \mathcal{B})^{\topp}$ is an $\text{LS}_{\text{\rm w}}$-space, and in particular regular.
Therefore, $(E, \mathcal{B}) = (E, \mathcal{B})^{\topp\hspace{-1pt}\bor}$.
Thus $(E, \mathcal{B}) \in \bLB$.

\smallskip

Particular examples in $\bLB$ are then the \emph{nuclear bornological }LB\emph{-spaces}, where in \eqref{eq:NuclearCond} we replace `weakly compact' with `trace-class'; they are precisely the bornologifications of nuclear LB-spaces. The latter are, depending on the literature, called LN-spaces, DFN-spaces \cite[p.~62]{Klaus} or DNF-spaces \cite[Section 8]{Scholze}.
\hfill\diam{}
\end{ex}

After this positive example, let us illustrate that in \eqref{eq:bLB=rangeInLBb} objects of $\bLB$ can very well be reached by locally convex spaces that do \emph{not} belong to $\LB$.

\begin{ex}\label{EX-LAST}There exists $(E, \tau) \in \Tc \setminus \LB$ such that $(E, \tau)^{\bor} \in \bLB$.
\end{ex}

\begin{proof}For any Fr\'echet space $(F, \mu)$ with a fundamental increasing sequence of norms $(\|\cdot\|_n)_{n \in \NN}$, we write
\[
B'_n = \bigl\{ f' \in F^\prime\:\big|\: |\langle f', f \rangle| \leqslant 1 \text{ for all } f \in F \text{ with } \|f\|_n \leqslant 1\bigr\}
\]
for $n\in\NN$. Consider $(F^\prime, b(F^\prime, F)) \in \Tc$, where $b(F^\prime, F)$ denotes the topology on $F^\prime$ of uniform convergence on bounded sets of $(F, \mu)$. Then, by \cite[Lem 25.5]{MV1997}, each $B'_n$ is a Banach disk of $(F^\prime, b(F^\prime, F))^{\bor}$ and we have
\[
(F^\prime, \mathcal{B}) := (F^\prime, b(F^\prime, F))^{\bor} = \colim_{n \in \NN} ((F^\prime)_{B'_n}, \|\cdot\|_{B'_n})^{\bor} \in \LBb .
\]
Hence $(F^\prime, \mathcal{B}) \in \bLB$ by Theorem \ref{t:bLB=rangeInLBb}.
 
\smallskip
   
But, if $F$ is not distinguished \cite[Dfn on p.~300]{MV1997}, that is, if $(F^\prime, b(F^\prime, F)) \notin \NTc$, then in particular $(F^\prime, b(F^\prime, F)) \notin \LB$. Thus, any such $(F, \mu)$ yields an example $(E,\tau)=(F',b(F',F))$ with the properties claimed; concrete $(F, \mu)$ can be found in \cite[Ex 27.19]{MV1997}.
\end{proof}

\begin{rmk}\label{rmk:nAgen?} It is interesting to consider which of the results in this section work in other settings, such as over more general non-trivially valued Banach fields. For example, over non-Archimedean non-trivially valued $k$, it is quite possible that Theorems \ref{thm-char}, \ref{t:RangeLB}, \ref{t:bLB=rangeInLBb}, and \ref{t:bLb=LBb+NBorn} hold when $k$ is spherically complete and we restrict the semi-norms on both the bornological and topological sides to themselves be non-Archimedean. Indeed, under this hypothesis, one still has the Hahn-Banach Theorem, see \cite[Section 4.1]{MR1977452}, and De Wilde's Theorem, see \cite[Thm 4.24]{MR3577218}. However, we have not verified the details.\hfill\diam{}
\end{rmk}

\section{Idempotent adjunctions and consequences}\label{SEC:4}

Many of the properties of Fact \ref{Facts} can be proven completely abstractly using the theory of idempotent adjunctions. Moreover, coupled with Theorem \ref{t:CharRegularity}, this interpretation allows us to prove some useful properties of the category $\bLB$. As will be clear, some of the results below have appeared in the analysis literature before. However, where possible, we give largely categorical proofs that will further generalise to categories of bornological modules defined over Banach rings (and in fact to many more abstract settings). Such categories have recently been used in developing the theory of derived analytic geometry in \cite{ben2024perspective, KKM, kelly2025localising}. The analytic and categorical properties of these bornological modules will be further elaborated upon in forthcoming work of the first author in \cite{Kellyconvenient}, in which a generalisation of the bornologification/topologification functors will also be discussed. 

\smallskip

There are many different ways to characterise what it means for an adjunction
\begin{equation}\label{eq:idempadjunction}
\begin{tikzcd}
L\hspace{2pt}\hspace{2pt}\colon \ensuremath{\text{\bf\textsf{C}}} \arrow[r, shift left=0.75ex]  &\arrow[l, shift right=-0.75ex, ]\ensuremath{\text{\bf\textsf{D}}}\hspace{2pt}\colon R,
\end{tikzcd}
\end{equation}
where $L$ is the left adjoint and $R$ the right adjoint, to be idempotent. For a full list, one can consult \cite[Prop 2.28]{MR667399} or \cite[Thm 3.8.7]{grandis2021category}. For the purposes of the present note, the most useful characterising property is that, for any object $D$ of $\ensuremath{\text{\bf\textsf{D}}}$, the component of the unit on $R(D)$,
$$\eta_{R(D)}\colon R(D)\rightarrow RLR(D)$$
is an isomorphism. 

Recall that a morphism in a category is called a \textit{regular monomorphism} if it is the equaliser of a parallel pair of morphisms \cite[Dfn 16.13]{MR2377903}.

\begin{prop}\label{prop:adjidempepi}
An adjunction as in \eqref{eq:idempadjunction} is idempotent if and only if for any object $D$ of $\ensuremath{\text{\bf\textsf{D}}}$ the morphism
$\eta_{R(D)}\colon R(D)\rightarrow RLR(D)$
is an epimorphism.
\end{prop}

\begin{proof}
   By standard properties of adjunctions, we have $\operatorname{id}_{R(D)}=R(\varepsilon_{D})\circ \eta_{R(D)}$, where $\varepsilon$ is the counit of the adjunction. In particular, $\eta_{R(D)}$ is a split, and hence regular, monomorphism, see \cite[Prop 16.15]{MR2377903}. Since it is also an epimorphism by assumption, it is an isomorphism by \cite[Prop 16.16]{MR2377903}.
\end{proof}

Under the stronger assumption that $\eta_{C}$ is an epimorphism for all $C\in\ensuremath{\text{\bf\textsf{C}}}$, the above result is proven in \cite[Section 2-C]{porst1991concrete}.

\smallskip

Suppose our adjunction \eqref{eq:idempadjunction} is idempotent. Let $\ensuremath{\text{\bf\textsf{C}}}^{\eta}$ denote the full subcategory of $\ensuremath{\text{\bf\textsf{C}}}$ consisting of those objects $C$ for which the unit
$$\eta_{C}\colon C\rightarrow R L(C)$$
is an isomorphism. Similarly, let $\ensuremath{\text{\bf\textsf{D}}}^{\varepsilon}$ denote the full subcategory of $\ensuremath{\text{\bf\textsf{D}}}$ consisting of those objects $D$ for which the counit $\varepsilon_{D}\colon LR(D)\rightarrow D$ is an isomorphism. By  \cite[Thm 3.8.7]{grandis2021category}, the adjunction restricts to an adjoint equivalence
\begin{equation*}
\begin{tikzcd}
L\hspace{2pt}\hspace{2pt}\colon\ensuremath{\text{\bf\textsf{C}}}^{\eta} \arrow[r, shift left=0.75ex]  &\arrow[l, shift right=-0.75ex, ]\ensuremath{\text{\bf\textsf{D}}}^{\varepsilon}\hspace{2pt}\colon R.
\end{tikzcd}
\end{equation*}

Also by \cite[Thm 3.8.7]{grandis2021category}, the inclusion $\ensuremath{\text{\bf\textsf{C}}}^{\eta}\rightarrow\ensuremath{\text{\bf\textsf{C}}}$ has a left adjoint given by $C\mapsto RL(C)$, and the inclusion $\ensuremath{\text{\bf\textsf{D}}}^{\varepsilon}\rightarrow\ensuremath{\text{\bf\textsf{D}}}$ has a right adjoint given by $D\mapsto LR(D)$. Consequently, $\ensuremath{\text{\bf\textsf{C}}}^{\eta}$ is closed under limits in $\ensuremath{\text{\bf\textsf{C}}}$, and $\ensuremath{\text{\bf\textsf{D}}}^{\varepsilon}$ is closed under colimits in $\ensuremath{\text{\bf\textsf{D}}}$. 

\begin{ex}\label{ex:idempotentbt}
In \cite[Rmk 2.1.11]{MR961256}, Fr\"{o}licher and Kriegl consider the following specific case of Proposition \ref{prop:adjidempepi}. Suppose we have an adjunction
\begin{equation*}
\begin{tikzcd}
L\hspace{2pt}\colon \ensuremath{\text{\bf\textsf{C}}} \arrow[r, shift left=0.75ex]  &\arrow[l, shift right=-0.75ex, ]\ensuremath{\text{\bf\textsf{D}}}\hspace{2pt}\colon R,
\end{tikzcd}
\end{equation*}
between \textit{concrete} categories in which $L$ and $R$ both preserve the underlying sets and underlying maps. If the underlying maps of the components of unit and counit are the identity on underlying sets, then \cite[Rmk 2.1.11]{MR961256} states that the adjunction is idempotent. This claim is clearly implied by Proposition \ref{prop:adjidempepi}.

\smallskip

\noindent{}For example, let $k$ be a non-trivially valued Banach field. Consider the pair of functors \eqref{FUN} from Section \ref{SEC-2}:
\begin{equation*}
\begin{tikzcd}
(\,\cdot\,)^{\topp}\hspace{2pt}\colon \Born \arrow[r, shift left=0.75ex]  &\arrow[l, shift right=-0.75ex] \Tc\hspace{2pt}\colon (\,\cdot\,)^{\bor}.
\end{tikzcd}
\end{equation*}
As stated above, \cite[Prop 2.1.10]{MR961256} says that, over $k\in\{\mathbb{R},\mathbb{C}\}$, this pair of functors forms an adjoint pair. By \cite[Section 3.2]{BBBK18}, following \cite{Houzel}, this adjunction in fact exists over any non-trivially valued Banach field $k$. Apart from some points where we explicitly state otherwise, for the rest of this section we will work in this generality. Let $(E,\tau) \in \Tc$. Then the underlying vector spaces of $(E,\tau)^{\bor}$ and $(E,\tau)^{\bor\hspace{-1pt}\topp\hspace{-1pt}\bor}$ are both still $E$, and on underlying vector spaces the morphism 
$$(E,\tau)^{\bor}\rightarrow(E,\tau)^{\bor\hspace{-1pt}\topp\hspace{-1pt}\bor}$$
is just the identity. The category $\Born^{\eta}$ is exactly $\NBorn$, and the category $\Tc^{\varepsilon}$ is $\NTc$. 
Note that \cite{BBBK18} imposes the requirement that the semi-norms on both the bornological and topological sides should themselves be non-Archimedean. Choosing whether or not to impose this restriction has no ramifications for most of the arguments below, except in Remark \ref{rmk:NLBRLB}. There is also no effect on the standard proof that the pair above forms an adjoint pair, but for the sake of completeness we include it after this example.  \hfill\diam{}
\end{ex}

\begin{prop}
    Let $k$ be a non-trivially valued Banach field. Then there exists an adjunction
    \begin{equation*}
    \begin{tikzcd}
(\,\cdot\,)^{\topp}\hspace{2pt}\colon \Born \arrow[r, shift left=0.75ex]  &\arrow[l, shift right=-0.75ex] \Tc\hspace{2pt}\colon (\,\cdot\,)^{\bor}.
\end{tikzcd}
\end{equation*}
\end{prop}
\begin{proof}
    Let $(E,\mathcal{B})\in \Born$ and $(F,\tau)\in\Tc$. Let $\mathcal{P}$ be a family of seminorms on $F$ determining the topology $\tau$. Let $f:E\rightarrow F$ be a linear morphism between the underlying vector spaces. It is enough to show that $f\colon(E,\mathcal{B})\rightarrow(F,\tau)^{\bor}$ is bounded if and only if $f\colon(E,\mathcal{B})^{\topp}\rightarrow (F,\tau)$ is continuous. 
    
    Suppose that $f\colon(E,\mathcal{B})\rightarrow(F,\tau)^{\bor}$ is bounded. Fix $\rho\in\mathcal{P}$ and $r>0$, and consider the open ball $B_{\rho}(0,r)$ of radius $r$. Note that $f^{-1}(B_{\rho}(0,r))$ is evidently absolutely convex. We claim that it is bornivorous. Let $B\in\mathcal{B}$. Then $f(B)$ is bounded in $(F,\tau)$. In particular, $\rho$ is a bounded function on $f(B)$. There is therefore some $\delta>0$ such that $f(B)\subseteq B_{\rho}(0,\delta)$. Since $k$ is non-trivially valued, we can choose $\lambda\in k$ such that $0<|\lambda|<\frac{r}{\delta}$. Then $\lambda B_{\rho}(0,\delta)\subset B_{\rho}(0,r)$, and $B_{\rho}(0,\delta)\subset\frac{1}{\lambda}B_{\rho}(0,r)$. Therefore, $B\subseteq f^{-1}(\frac{1}{\lambda}B_{\rho}(0,r))=\frac{1}{\lambda}(f^{-1}(B_{\rho}(0,r)))$. Hence, $f\colon(E,\mathcal{B})^{\topp}\rightarrow (F,\tau)$ is continuous. 
    
    Conversely, suppose that $f\colon(E,\mathcal{B})^{\topp}\rightarrow (F,\tau)$ is continuous. Let $B\in\mathcal{B}$. We will show that $\rho|_{f(B)}$ is bounded for all $\rho\in\mathcal{P}$. Indeed, let $r>0$ and consider $B_{\rho}(0,r)$. Then there is an absolutely convex, bornivorous set $U\subseteq f^{-1}(B_{\rho}(0,r))$ containing $0$. Let $\lambda\in k\setminus\{0\}$ be such that $B\subset\lambda U$. Then $f(B)\subseteq\lambda f(U)\subseteq \lambda B_{\rho}(0,r)=B_{\rho}(0,|\lambda| r)$. Therefore $\rho|_{f(B)}$ is bounded by $|\lambda| r$. Since $\rho$ was arbitrary, this proves our claim. Hence, $f:(E,\mathcal{B})\rightarrow (F,\tau)^{\bor}$ is bounded. 
\end{proof}

\begin{rmk}\label{rmk:NLBRLB} Let $k\in\{\mathbb{R},\mathbb{C}\}$. As seen in Theorem \ref{t:bLb=LBb+NBorn}, we have $\bLB = \LBb \cap \NBorn$.
Also, any $(E, \tau) \in \LBr$ is of the form $(E, \mathcal{B})^{\topp}$ for $(E, \mathcal{B}) \in \bLB$, since $(E, \tau) = (E, \tau)^{\bor\hspace{-1pt}\topp}$.
Consequently, the equivalence $\NTc\cong\NBorn$ restricts to an equivalence $\bLB\cong\LBr$. As mentioned in Remark \ref{rmk:nAgen?}, it is possible that Theorem \ref{t:bLb=LBb+NBorn} holds for non-Archimedean bornological and locally convex spaces over a non-Archimedean, spherically complete, non-trivially valued field $k$. It is likely that the restriction to non-Archimedean semi-norms here is crucial. Indeed, Federico Bambozzi has pointed out to the authors that for $k$ non-Archimedean, $k$ is not injective in the category of all Banach $k$-vector spaces relative to the quasi-abelian exact structure. In particular, there is no Hahn-Banach Theorem, so the proof of Theorem \ref{thm-char} would not immediately generalise.\hfill\diam{}
\end{rmk}

Let us now consider the question of subobject stability of the category $\ensuremath{\text{\bf\textsf{C}}}^{\eta}$ in the context of an idempotent adjunction. Later, we will use our results to deduce that $\bLB$ is stable under strict and bornologically closed subobjects in $\Born$. 
 
\smallskip
Let $\mathcal{S}$ be a class of morphisms in a category $\ensuremath{\text{\bf\textsf{C}}}$ that contains all isomorphisms and is closed under taking compositions. The class $\mathcal{S}$ is said to \textit{satisfy the left (resp.\ right) cancellation property} if whenever $g\circ f$ is a map in $\mathcal{S}$ then $f$ (resp.\ $g$) is in $\mathcal{S}$. Typically we will require that $\mathcal{S}$ consists of regular monomorphisms in the left cancellation case.

\begin{ex}\label{ex:leftcancelexact}Let $\ensuremath{\text{\bf\textsf{E}}}$ be a weakly idempotent complete exact category (with the specified class of short exact sequences being implicit). Let $\ensuremath{\text{\bf\textsf{AdMon}}}$ denote the class of admissible monomorphisms for the exact structure. Such morphisms are by definition kernels, so they are regular. Moreover, by the definition of an exact category, $\ensuremath{\text{\bf\textsf{AdMon}}}$ is closed under composition. By the Obscure Axiom \cite[Props 2.16 and 7.6]{Buehler}, $\ensuremath{\text{\bf\textsf{AdMon}}}$ satisfies the left cancellation property. \hfill\diam{}
\end{ex}

\begin{prop}\label{prop:cancelsub}
Let
\begin{equation*}
    \begin{tikzcd}
L\hspace{2pt}\colon \ensuremath{\text{\bf\textsf{C}}} \arrow[r, shift left=0.75ex]  &\arrow[l, shift right=-0.75ex, ]\ensuremath{\text{\bf\textsf{D}}}\hspace{2pt}\colon R
\end{tikzcd}
\end{equation*}
be an idempotent adjunction such that for any $C\in\ensuremath{\text{\bf\textsf{C}}}$ the map $\eta_{C}\colon C\rightarrow RL(C)$ is an epimorphism. Let $\mathcal{S}$ be a class of regular monomorphisms that contains the class of isomorphisms, satisfies the left cancellation property, and is stable by composition. If $i\colon C\rightarrow D$ is in $\mathcal{S}$ and $D\in\ensuremath{\text{\bf\textsf{C}}}^{\eta}$ then $C\in\ensuremath{\text{\bf\textsf{C}}}^{\eta}$.
\end{prop}

\begin{proof}
Consider the commutative diagram

\[
\begin{tikzcd}
C \arrow{r}{\eta_{C}} \arrow[swap]{d}{i} & RL(C) \arrow{d}{RL(i)} \\
D  \arrow{r}{\eta_{D}} & RL(D)
\end{tikzcd}
\]
Since $D$ is in $\ensuremath{\text{\bf\textsf{C}}}^{\eta}$, $\eta_{D}$ is an isomorphism and so is in $\mathcal{S}$. The map $i$ is in $\mathcal{S}$ by assumption. Hence, the composition of the left-hand map and the bottom map is in $\mathcal{S}$. By the left cancellation property, $\eta_{C}$ is in $\mathcal{S}$. In particular, it is a regular monomorphism. By assumption, it is an epimorphism. Therefore it is an isomorphism.
\end{proof}

The category $\Born$ is quasi-abelian. Equivalently (or indeed by definition), this means that the class of all kernel-cokernel pairs defines an exact structure on $\Born$. In the context of quasi-abelian exact structures, admissible monomorphisms are often called \textit{strict monomorphisms}. In the category $\Born$, a morphism $f\colon(E,\mathcal{B})\rightarrow (F,\mathcal{C})$ is a strict monomorphism, precisely if it is injective, and $\mathcal{B}=\{f^{-1}(C)\:|\:C\in\mathcal{C}\}$. In other words, $f$ induces an isomorphism between $(E,\mathcal{B})$ and $(f(E),\mathcal{C}\cap f(E))$. See \cite{PS00} for details and note that Prosmans and Schneiders work in the $k=\mathbb{C}$ setting, but the proofs work identically over a general non-trivially valued Banach field.

\smallskip

Recall from \cite[Dfns 3.1--3.3]{BBBK18} that  a sequence $(x_{n})_{n\in\mathbb{N}}$ in a bvs $(E,\mathcal{B})$ is said to \textit{converge to }$x\in E$ \textit{in the sense of Mackey}, if
$$
\exists\:B\in\mathcal{B}\;\forall\:\lambda\in k^{\times}\;\exists\:n_0\in\NN\;\forall\:n\geqslant n_0\colon x_{n}-x\in\lambda B.
$$
A subset $U\subseteq E$ is said to be \textit{bornologically closed} if whenever a sequence of elements $(x_{n})_{n \in \NN}$ of $U$ converges to an element $x$ of $E$ in the sense of Mackey, then $x\in U$. There is also the notion of a \textit{Cauchy-Mackey sequence}. A bvs is then called \textit{semi-complete}, if every Cauchy-Mackey sequence is convergent in the sense of Mackey. As mentioned in \cite[after Dfn 3.3]{BBBK18}, complete bvs are semi-complete, but the converse in general does not hold.

\smallskip

\label{PRES}According to Fact \ref{Facts}(vii) the category $\Born$ is equivalent to the subcategory of $\ensuremath{\text{\bf\textsf{Ind(SNrm)}}}$ consisting of essentially monomorphic objects. We call an isomorphism $(E,\mathcal{B})\cong\colim_{i\in I}(E_{i},\|\cdot\|_{i})^{\bor}$, with each $(E_{i},\|\cdot\|_{i})$ being a semi-normed space and each linking map $(E_{i},\|\cdot\|_{i})^{\bor}\rightarrow (E_{i'},\|\cdot\|_{i'})^{\bor}$ in the diagram being a monomorphism, a \textit{presentation of} $(E,\mathcal{B})$. Note that in such a presentation, we have an isomorphism of underlying vector spaces $E\cong\colim_{i\in I}E_{i}$.


\smallskip

Let $f\colon(E,\mathcal{B})\rightarrow (F,\mathcal{C})$ be a morphism of bornological spaces, and let $U\subseteq F$ be bornologically closed. Note that if $(x_{n})_{n\in\NN}$ converges to $x$ in $(E,\mathcal{B})$ in the sense of Mackey, then $(f(x_{n}))_{n\in\NN}$ converges to $f(x)$ in $(F,\mathcal{C})$ in the sense of Mackey. Then, as in the case $k=\mathbb{C}$ discussed in \cite[Rmk 2.11(3)]{HN}, $f^{-1}(U)$ is bornologically closed in $(E,\mathcal{B})$.


\begin{prop}\label{prop:convmackeyequiv} Let $(E,\mathcal{B})\cong\colim_{i\in I}(E_{i},\|\cdot\|_{i})^{\bor}$ be a presentation of a bvs. We identify each $E_{i}$ with its range in $E$ as linear spaces. Let $(x_{m})_{m\in\NN}$ be a sequence in $(E,\mathcal{B})$. Then $(x_{m})_{m\in\NN}$ converges to some $x\in E$ in the sense of Mackey if and only if all elements $x_{m}$ and $x$ are contained in some $E_{i}$, and $(x_{m})_{m\in\NN}$ converges to $x$ in the semi-normed space $(E_{i},\|\cdot\|_{i})$. 
\end{prop}

\begin{proof}
Identifying each $(E_{i},\|\cdot\|_{i})^{\bor}$ with a subspace of $(E,\mathcal{B})$, the set $\{B_{E_{i}}(0,\delta)\:|\:\delta>0,\,i\in I\}$  of open balls forms a basis for the bornology $\mathcal{B}$ \--- that is any bounded subset $B\in\mathcal{B}$ is contained in some open ball of finite positive radius in one of the $(E_{i},\|\cdot\|_{i})$. 

\smallskip

\noindent{}Suppose that $(x_{m}-x)_{m\in\mathbb{N}}$ converges to $0$ in the sense of Mackey. By \cite[Rmk 3.4]{BBBK18} (or \cite[Prop 1.4(1)]{HN} in the case $k=\mathbb{C}$), there exists a bounded disk $B\subseteq E$ such that $(x_{m}-x)_{m\in\mathbb{N}}$ is contained in the semi-normed space $E_{B}$, and $(x_{m}-x)_{m\in\mathbb{N}}$ converges to zero in the semi-normed space $E_{B}$. By passing to a larger disk if necessary, we may assume that $x\in E_{B}$. Hence the sequence $(x_{m})_{m\in\mathbb{N}}$ is also contained in $E_{B}$. Now, $B\subseteq B_{E_{i}}(0,\delta)$ for some $i$ and some $\delta>0$. Hence $E_{B}\subseteq E_{B_{E_{\scalebox{0.55}{\hspace{1pt}$i$}}}(0,\delta)}=(E_{i},\|\cdot\|_{i})$ as linear spaces with continuous maps from left to right. Thus $\{x_{m}\}_{m\in\mathbb{N}}\cup\{x\}$ is contained in $E_{i}$, and $(x_{m}-x)_{m\in\mathbb{N}}$ converges to $0$ in $(E_{i},\|\cdot\|_{i})$. Hence $(x_{m})_{m\in\mathbb{N}}$  converges to $x$ in $(E_{i},\|\cdot\|_{i})$.

\smallskip

\noindent{}Conversely, if the sequence $(x_{m})_{m\in\mathbb{N}}$ is contained in some $(E_{i},\|\cdot\|_{i})^{\bor}$ and converges in $(E_{i},\|\cdot\|_{i})$ to $x$, then we claim that $(x_{m})_{m\in\mathbb{N}}$ converges to $x$ in $(E,\mathcal{B})$ in the sense of Mackey. Indeed, since the bounded morphism $(E_{i},\|\cdot\|_{i})^{\bor}\rightarrow (E,\mathcal{B})$ sends convergent sequences to convergent sequences (in the sense of Mackey), we may assume that $(E,\mathcal{B})=(E,\|\cdot\|)^{\bor}$ is semi-normed. Let $B=B_{E}(0,1)$ be the open ball of radius $1$. For all $\gamma\in k^{\times}$ there is $m_{0}$ such that $m\geqslant m_{0}$ implies that $\|x-x_{m}\|<|\gamma|$, i.e., $x_{m}-x\in\gamma B_{E}(0,1)$. In other words, $(x_{m})_{m\in\mathbb{N}}$ converges to $x$ in the sense of Mackey.
\end{proof}

Implicit in the proof of Proposition \ref{prop:convmackeyequiv} is the following fact, which actually holds for bornologifications of metrisable locally convex spaces by \cite[Prop 1.4(3)]{HN} over $\mathbb{R}$ or $\mathbb{C}$.

\begin{cor}\label{cor:exbanclosed}
    Let $(E,\|\cdot\|)^{\bor}$ be the bornologification of a semi-normed space $(E,\|\cdot\|)$. A sequence $(x_{m})_{m\in\mathbb{N}}$ converges to $x\in E$ in the sense of Mackey if and only if it converges to $x$ topologically. In particular, a subset $U\subseteq E$ is bornologically closed if and only if it is topologically closed.\hfill\diam{}
\end{cor}

\begin{cor}\label{cor:colimclosed}
    Let $(E,\mathcal{B})\cong\colim_{i\in I}(E_{i},\|\cdot\|_{i})^{\bor}$ and $(F,\mathcal{C})\cong\colim_{i\in I}(F_{i},\|\cdot\|_{i})^{\bor}$ be presentations of bornological spaces. Let $f\colon(E,\mathcal{B})\rightarrow(F,\mathcal{C})$ be a strict morphism that restricts to bounded strict morphisms $f_{i}\colon(E_{i},\|\cdot\|_{i})^{\bor}\rightarrow(F_{i},\|\cdot\|_{i})^{\bor}$. If each $f_{i}$ has closed range then $f$ has closed range.
\end{cor}

\begin{proof}
Let $(x_{m})_{m\in\NN}$ be a sequence in $(E,\mathcal{B})$ such that $(f(x_m))_{m\in\NN}$ converges to some $y$ in $(F,\mathcal{C})$ in the sense of Mackey. Then by Proposition \ref{prop:convmackeyequiv}, all the $f(x_{m})$ and $y$ are contained in some $F_{i}$, and $(f(x_{m}))_{m\in\NN}$ converges to $y$ in $(F_{i},\|\cdot\|_{i})$. Moreover, the collection $\{f(x_{m})\}_{m\in\NN}\cup\{y\}$ is bounded. By strictness, and passing to a larger $i$ if necessary, we must have that $\{x_{m}\}_{m\in\NN}$ is contained in some $E_{i}$, and is bounded there. Since $f_{i}\colon(E_{i},\|\cdot\|_{i})^{\bor}\rightarrow (F_{i},\|\cdot\|_{i})^{\bor}$ has bornologically closed range, and hence topologically closed range, we must have $y = f(x)$ for some $x\in E_{i}$. This completes the proof.
\end{proof}

Over the course of the next few results, we are effectively going to prove a converse to Corollary \ref{cor:colimclosed}.

\begin{prop}\label{prop:pullback}
 Let $g\colon(E,\mathcal{B})\rightarrow (G,\mathcal{D})$, and $f\colon(F,\mathcal{C})\rightarrow(G,\mathcal{D})$ be morphisms in $\Born$. If $g$ is a strict monomorphism then so is the pullback map $g'\colon(E,\mathcal{B})\times_{(G,\mathcal{D})}(F,\mathcal{C})\rightarrow (F,\mathcal{C})$. Moreover, if $g$ has closed range then so does $g'$.
\end{prop}

\begin{proof}

The underlying set of $(E,\mathcal{B})\times_{(G,\mathcal{D})}(F,\mathcal{C})$ is $E\times_{G}F$. In particular, if the underlying morphism of $g$ is the inclusion of a subset, then the underlying set of the fibre-product is in bijection with $f^{-1}(g(E))$. 

\smallskip

In any additive category with finite limits, kernels are stable by pullback. Thus $g'$ is automatically strict whenever $g$ is.    In particular, $g'$ induces an isomorphism between $(E,\mathcal{B})\times_{(G,\mathcal{D})}(F,\mathcal{C})$ and its range under $g'$, which is $f^{-1}(g(E))$ equipped with the subspace bornology.
 
\smallskip

Finally, as mentioned above, the preimage of a bornologically closed subspace is closed. Hence if $g(E)$ is bornologically closed in $G$, then $f^{-1}(g(E))$ is bornologically closed in $F$.
\end{proof}

\begin{prop}\label{prop:pullbackpres}
    Let $(F, \mathcal{C}) \cong \colim_{i\in I}(F_{i}, \|\cdot\|_{i})^{\bor}$ be a presentation of a bvs $(F,\mathcal{C})$. Let $(E,\mathcal{B})\rightarrow (F,\mathcal{C})$ be any morphism. Then the natural map 
    $$\colim_{i\in I}(E,\mathcal{B})\times_{(F,\mathcal{C})}(F_{i},\|\cdot\|_{i})^{\bor}\rightarrow (E,\mathcal{B})$$
    is an isomorphism, and each linking map $(E,\mathcal{B})\times_{(F,\mathcal{C})}(F_{i},\|\cdot\|_{i})^{\bor}\rightarrow (E,\mathcal{B})\times_{(F,\mathcal{C})}(F_{i'},\|\cdot\|_{i'})^{\bor}$ is a monomorphism.
\end{prop}

\begin{proof}
Let $\kappa$ be an uncountable regular cardinal such that each of $E$, $F$, $E_{i}$, and $I$ are $\kappa$-small sets. Let $\Born^{\pmb\kappa}\subseteq\Born$ denote (a small skeleton of) the full subcategory consisting of bvs whose underlying set is $\kappa$-small. Then $\Born^{\pmb\kappa}$ is closed under $\kappa$-small colimits and finite limits in $\Born$. Thus, we may work in $\Born^{\pmb\kappa}$.  Now, $\Born^{\pmb\kappa}\subseteq\ensuremath{\text{\bf\textsf{Ind(SNrm}}}^{\pmb\kappa}\ensuremath{\text{\bf\textsf{)}}}$ is closed under $\kappa$-small monomorphic colimits, where $\SNrm^{\pmb\kappa}$ is the category of semi-normed spaces whose underlying set is $\kappa$-small. 
The category $\SNrm^{\pmb\kappa}$ is small and finitely cocomplete. Therefore, $\ensuremath{\text{\bf\textsf{Ind(SNrm}}}^{\pmb\kappa}\ensuremath{\text{\bf\textsf{)}}}$ is locally finitely presentable by \cite[Thm 5.5(iv)]{MR1935970}. Consequently, the natural map 
    $$\colim_{i\in I}(E,\mathcal{B})\times_{(F,\mathcal{C})}(F_{i},\|\cdot\|_{i})^{\bor}\rightarrow (E,\mathcal{B})$$
   is an isomorphism in $\ensuremath{\text{\bf\textsf{Ind(SNrm}}}^{\pmb\kappa}\ensuremath{\text{\bf\textsf{)}}}$ by \cite[Prop 1.59]{MR1294136}. It therefore suffices to prove that $E\times_{F}F_{i}\rightarrow E\times_{F}F_{j}$ is a monomorphism. Since monomorphisms are stable by pullback, $E\times_{F}F_{i}\rightarrow E$ is a monomorphism. This map factors through $E\times_{F}F_{j}\rightarrow F$. Hence $E\times_{F}F_{i}\rightarrow E\times_{F}F_{j}$ is also a monomorphism.
\end{proof}


\begin{prop}\label{prop:closedb} Let $(F,\mathcal{C})=(F,\|\cdot\|_F)^{\bor}\in\Born$, where $(F,\|\cdot\|_F)$ is a semi-normed space and let $(E,\mathcal{B})\in\Born$ be arbitrary.\vspace{3pt}

\begin{compactitem}
\item[(i)] Let $f\colon(E,\mathcal{B})\rightarrow (F,\|\cdot\|_F)^{\bor}$ be a strict monomorphism. Then $(E,\mathcal{B})$ is isomorphic to $(f(E),\|\cdot\|_{F}|_{f(E)})^{\bor}$. If $(F,\|\cdot\|_F)$ is normed then so is $(f(E),\|\cdot\|_{F}|_{f(E)})$. \vspace{3pt}

\item[(ii)] Assume now that $(F,\|\cdot\|_F)$ is a Banach space. Let $f\colon(E,\mathcal{B})\rightarrow (F,\|\cdot\|_F)^{\bor}$ be a strict monomorphism with bornologically closed range. Then $(E,\mathcal{B})$ is isomorphic to $(f(E),\|\cdot\|_{F}|_{f(E)})^{\bor}$, and $(f(E),\|\cdot\|_{F}|_{f(E)})$ is a Banach space.
\end{compactitem}
   
\end{prop}

\begin{proof} (i) As $f$ is a strict monic, \cite[Cor 1.7]{PS00} (or rather, a straightforward generalisation thereof to non-trivially valued $k$) implies that $f\colon(E,\mathcal{B})\rightarrow(f(E),\mathcal{C}\cap f(E))$ is an isomorphism. Let $B \subseteq f(E)$. Then $B$ is bounded in $(f(E),\mathcal{C}\cap f(E))$ if and only if $B$ is $\|\cdot\|_F$-bounded. 
The second statement is clear.
\smallskip

(ii) By the first part, $(f(E),\|\cdot\|_{F}|_{f(E)})^{\bor}$ is normed. Being  
bornologically closed in $(F,\|\cdot\|_F)^{\bor}$, $f(E)$ is topologically closed in $(F,\|\cdot\|_F)$ by Corollary \ref{cor:exbanclosed}. In particular, $(f(E),\|\cdot\|_{F}|_{f(E)})$ is Banach. 
\end{proof}

\begin{cor}\label{cor:presclosed}
    Let $(F,\mathcal{C})\cong\colim_{i\in I}(F_{i},\|\cdot\|_{i})^{\bor}$ be a presentation of $(F,\mathcal{C})$. Let $(E,\mathcal{B})\subseteq (F,\mathcal{C})$ be a bornological subspace. Denote by $f\colon(E,\mathcal{B})\rightarrow(F,\mathcal{C})$ the inclusion morphism. Then the following hold. \vspace{3pt}
    \begin{compactitem}
        \item[(i)] The isomorphism $(E,\mathcal{B})\cong\colim_{i\in I}(E,\mathcal{B})\times_{(F,\mathcal{C})}(F_{i},\|\cdot\|_{i})^{\bor}$ from Proposition \ref{prop:pullbackpres} is a presentation of $(E,\mathcal{B})$.
        \vspace{3pt} 
        \item[(ii)] The induced morphisms $f_{i}\colon(E,\mathcal{B})\times_{(F,\mathcal{C})}(F_{i},\|\cdot\|_{i})^{\bor}\rightarrow (F_{i},\|\cdot\|_{i})^{\bor}$ are strict monomorphisms.
   \item [(iii)] If $f$ has closed range then so does each $f_{i}$. \vspace{3pt}
    \item[(iv)] If $f$ has closed range and each $(F_{i},\|\cdot\|_{i})$ is Banach, then each $(E,\mathcal{B})\times_{(F,\mathcal{C})}(F_{i},\|\cdot\|_{i})^{\bor}$ is the bornologification of a Banach space. \vspace{3pt}
    \end{compactitem}
In particular, $(E,\mathcal{B})$ is bornologically closed in $(F,\mathcal{C})\cong\colim_{i\in I}(F_{i},\|\cdot\|_{i})^{\bor}$ if and only if there is a presentation $(E,\mathcal{B})\cong\colim_{i\in I}(E_{i},\|\cdot\|_{i})^{\bor}$, such that $f$ restricts to a strict monomorphism $f_{i}\colon(E_{i},\|\cdot\|_{i})^{\bor}\rightarrow (F_{i},\|\cdot\|_{i})^{\bor}$ with closed range for each $i\in I$. 
\end{cor}

\begin{proof}
    The four enumerated parts follow immediately from Proposition \ref{prop:pullbackpres}, the fact that strict monomorphisms are stable under pullback, and Proposition \ref{prop:closedb}. 
    
\smallskip

\noindent{}For the final claim, note that if $(E,\mathcal{B})$ is bornologically closed in $(F,\mathcal{C})$, then the required presentation of $(E,\mathcal{B})$ is $(E,\mathcal{B})\cong\colim_{i\in I}(E,\mathcal{B})\times_{(F,\mathcal{C})}(F_{i},\|\cdot\|_{i})^{\bor}$. The converse is Corollary \ref{cor:colimclosed}.
\end{proof}

\begin{cor}\label{cor:subLB}
    Let $(F,\mathcal{C})\cong \colim_{n\in\mathbb{N}}(F_{n},\|\cdot\|_{n})^{\bor}$ be an object of $\LBb$, and let $(E,\mathcal{B})\rightarrow (F,\mathcal{C})$ be a strict monomorphism with closed range. Then $(E,\mathcal{B})$ is an object of $\LBb$.
\end{cor}

\begin{proof}
    Corollary \ref{cor:presclosed} furnishes an explicit presentation $$(E,\mathcal{B})\cong\colim_{n\in\mathbb{N}}(E,\mathcal{B})\times_{(F,\mathcal{C})}(F_{n},\|\cdot\|_{n})^{\bor}$$ with each $(E,\mathcal{B})\times_{(F,\mathcal{C})}(F_{n},\|\cdot\|_{n})^{\bor}$ being the bornologification of a Banach space.
\end{proof}

Corollary \ref{cor:presclosed} may be considered a refinement of \cite[Lem 3.14]{BBBK18} (itself a special case of \cite[Prop 3.2(1)]{HN}), as the following shows.

\begin{cor}\label{cor:bornclosedcomplete}
    Let $(F,\mathcal{C})$ be a complete bornological space and $(E,\mathcal{B})\subseteq (F,\mathcal{C})$. Then $(E,\mathcal{B})$ is complete if and only if it is bornologically closed in $(F,\mathcal{C})$.
\end{cor}

\begin{proof}
    Since $(F,\mathcal{C})$ is complete, it has a presentation $(F,\mathcal{C})\cong\colim_{i\in I}(F_{i},\|\cdot\|_{i})^{\bor}$ with each $(F_{i},\|\cdot\|_{i})$ being a Banach space.

\smallskip

    \textquotedblleft{}$\Longleftarrow$\textquotedblright{} If $(E,\mathcal{B})$ is bornologically closed in $(F,\mathcal{C})$ then, by Corollary \ref{cor:presclosed},
    \begin{equation}\label{PR}
    (E,\mathcal{B})\cong\colim_{i\in I}(E,\mathcal{B})\times_{(F,\mathcal{C})}(F_{i},\|\cdot\|_{i})^{\bor}
    \end{equation}
    is a presentation with each $(E,\mathcal{B})\times_{(F,\mathcal{C})}(F_{i},\|\cdot\|_{i})^{\bor}$ the bornologification of a Banach space.
    
\smallskip

\textquotedblleft{}$\Longrightarrow$\textquotedblright{} Conversely, suppose that $(E,\mathcal{B})$ is complete. We still have the presentation \eqref{PR}. By \cite[p.\ 27--28]{BOBB}, $\CBorn$ is a reflective subcategory of $\Born$, and hence is closed under limits. Thus $(E,\mathcal{B})\times_{(F,\mathcal{C})}(F_{i},\|\cdot\|_{i})^{\bor}$ is complete. Since $(E,\mathcal{B})\times_{(F,\mathcal{C})}(F_{i},\|\cdot\|_{i})^{\bor}\cong (E_{i},\|\cdot\|_{i})^{\bor}$ with $(E_{i},\|\cdot\|_{i})$ semi-normed, we must have in fact that $(E_{i},\|\cdot\|_{i})$ is Banach. Now, each $(E_{i},\|\cdot\|_{i})^{\bor}\rightarrow (F_{i},\|\cdot\|_{i})^{\bor}$ is strict. Hence the semi-normed space $(E_{i},\|\cdot\|_{i})$ is isomorphic to its range regarded as a subspace of $(F_{i},\|\cdot\|_{i})$. Since this range is Banach, it is complete, and hence must be topologically closed in $(F_{i},\|\cdot\|_{i})$. Thus it is bornologically closed, by Corollary \ref{cor:exbanclosed}. We can now conclude by Corollary \ref{cor:presclosed}.
\end{proof}

We are now going to show that if $(E,\mathcal{B})\rightarrow (F,\mathcal{C})$ is a strict monomorphism with bornologically closed range, and $(F,\mathcal{C})\in\bLB$, then $(E,\mathcal{B})\in\bLB$. In fact, we shall show something more. Over $\mathbb{R}$ or $\mathbb{C}$, under the equivalence $\NBorn\cong\Precon$ mentioned in Section \ref{SEC-2} before Facts \ref{Facts}, those preconvenient vector spaces that are complete as bornological vector spaces are called \textit{convenient vector spaces} \cite[Dfn 2.6.3]{MR961256}. The full subcategory of $\Precon$ consisting of convenient vector spaces is denoted $\Con$, and the equivalence $\NBorn\cong\Precon$ restricts to an equivalence 
    \[ \NBorn\cap\CBorn\cong\Con . \] 
By abuse of terminology, we will say that a bornological space $(E,\mathcal{B})$ is \textit{convenient} if it belongs to the subcategory $\NBorn\cap\CBorn\subset\Born$. In particular, any object of $\bLB\subset\NBorn\cap\CBorn$ is convenient.

\begin{thm}\label{thm:NBornsub}
    Suppose that  $f\colon(E,\mathcal{B})\rightarrow (F,\mathcal{C})$ is a strict monomorphism in $\Born$. \vspace{2pt}
    \begin{compactitem}
        \item[(i)] If $(F,\mathcal{C})$ is in $\NBorn$ then so is $(E,\mathcal{B})$.\vspace{2pt}
        \item[(ii)] If $f$ additionally has closed range, and $(F,\mathcal{C})$ is convenient, then $(E,\mathcal{B})$ is convenient.\vspace{2pt}
        \item[(iii)] Suppose that $k\in\{\mathbb{R},\mathbb{C}\}$, that $f$ additionally has closed range, and that $(F,\mathcal{C})\in\bLB$. Then $(E,\mathcal{B})\in\bLB$.
    \end{compactitem}
\end{thm}

\begin{proof} (i) This follows immediately from Example \ref{ex:idempotentbt}, Example \ref{ex:leftcancelexact}, and Proposition \ref{prop:cancelsub}. 

\smallskip
 
(ii) This follows from the first part, and Corollary \ref{cor:bornclosedcomplete}.

\smallskip

(iii) Since $\bLB=\LBb\cap\NBorn$ when $k\in\{\mathbb{R},\mathbb{C}\}$, this follows from the first part, and Corollary \ref{cor:subLB}.
\end{proof}

\begin{rmk} Let $k\in\{\mathbb{R},\mathbb{C}\}$. Let $f\colon(E,\mathcal{B})\rightarrow (F,\mathcal{C})$ be a strict monomorphism in $\Born$ with $(F,\mathcal{C})\in\NBorn$, so that $(E,\mathcal{B})$ is then in $\NBorn$. Under the equivalence $\NBorn\cong\Precon$, such strict monomorphisms coincide with $\Precon$-embeddings in the sense of \cite[Dfn 2.4.6]{MR961256}. In this language, Theorem \ref{thm:NBornsub}(i) can be understood as a generalisation of \cite[Rmk 2.4.7]{MR961256} to more general Banach fields. Again for $k\in\{\mathbb{R},\mathbb{C}\}$, when $(F,\mathcal{C})$ is also complete as a bornological vector space and $f$ has closed range, Theorem \ref{thm:NBornsub}(ii) is equivalent to a combination of \cite[Rmk 2.4.7]{MR961256}, and the claim in \cite[Prop 2.6.4]{MR961256} that a preconvenient space admitting a $\Precon$-embedding into a convenient space with bornologically closed range (or $M$-\textit{closed} range in the terminology of loc.\ cit.) is convenient. Note that our proof strategy for Theorem \ref{thm:NBornsub}(i) using Proposition \ref{prop:cancelsub} is purely categorical and quite general. Indeed, it works over general non-trivially valued fields $k$, and the next two remarks suggest it is useful in even greater generality.\hfill\diam 
\end{rmk}

\begin{rmk}
    The category $\Tc$ is also quasi-abelian. Entirely dually (in the categorical sense) to Theorem \ref{thm:NBornsub}, using Proposition \ref{prop:cancelsub} one can show that if $g\colon(X,\tau_{X})\rightarrow (Y,\tau_{Y})$ is a strict epimorphism (equivalently, $g$ is a surjection and $\tau_{Y}$ is the locally convex quotient topology), and $(X,\tau_X)\in\NTc$, then $(Y,\tau_{Y})\in\NTc$.\hfill\diam 
\end{rmk}

We mentioned at the beginning of this section that some results of this section generalise to the setting of bornological modules over Banach rings; see, e.g., \cite{ben2024perspective} for a definition of the latter and further background to our concluding remark below.

\begin{rmk}\label{rmk:banring} Let $R$ be a Banach ring.  By definition, the category $\Born_{R}$ of bornological $R$-modules is the essentially monomorphic subcategory of the ind-completion $\ensuremath{\text{\bf\textsf{Ind(SNrm}}_{R}\textsf{)}}$ of the category $\ensuremath{\text{\bf\textsf{SNrm}}_{R}}$ of semi-normed $R$-modules with morphisms being bounded maps of $R$-modules. Similarly, $\CBorn_{R}$ is defined by replacing $\ensuremath{\text{\bf\textsf{SNrm}}}_{R}$ with the subcategory $\Ban_{R}$ of complete normed $R$-modules. Modulo some set-theoretic caution as in Proposition \ref{prop:pullbackpres}, more generally one can start with any (small) quasi-abelian category $\ensuremath{\text{\bf\textsf{E}}}$, and consider the full subcategory of essentially monomorphic objects $\Born\ensuremath{\text{\bf\textsf{(E)}}}\subseteq\ensuremath{\text{\bf\textsf{Ind(E}}\textsf{)}}.$ Of the results above, Proposition \ref{prop:pullbackpres} can be generalised to this abstract setting. 
    Corollary \ref{cor:presclosed} (i) and (ii) can be generalised as long as $\ensuremath{\text{\bf\textsf{E}}}$ is closed under strict subobjects in $\Born\ensuremath{\text{\bf\textsf{(E)}}}$ \--- notice that in the setting of the present paper, subobject stability is guaranteed by Proposition \ref{prop:closedb}(i). Moreover one can define internal LB-spaces in $\Born\ensuremath{\text{\bf\textsf{(E)}}}$ in the obvious way and, assuming subobject stability of $\ensuremath{\text{\bf\textsf{E}}}$,
an analogue of Corollary \ref{cor:subLB} holds. Corollary \ref{cor:colimclosed} can also be extended to more general Banach rings. In \cite{Kellyconvenient}, we will consider the generalisation of Theorem \ref{thm:NBornsub} to $\Born_{R}$ and $\CBorn_{R}$.
\end{rmk}

\vspace{5pt}

\begin{center}
{\sc Declaration on the Use of AI}
\end{center}

{

\small 
After this paper had been written the authors used ChatGPT (version Astra Extra-High) for proof-reading. The authors critically evaluated and revised all chatbot-made suggestions and take full responsibility for the mathematical results, proofs, references and final text.
\normalsize
}

\vspace{8pt}

\begin{center}
{\sc Acknowledgments}
\end{center}

{

\small 
The authors would like to thank the referee for their careful work which led to several improvements in the paper.

}


\end{document}